\documentclass[journal]{IEEEtran}

\usepackage[utf8]{inputenc}
\usepackage{xcolor}
\usepackage{xspace}
\usepackage{textcomp}

\usepackage{multirow}

\usepackage{pifont}

\usepackage{amsmath,amssymb,amsthm}
\usepackage{bm,fixmath} % for math bold letters

\usepackage{graphicx}
\graphicspath{{./images/}}
\usepackage[font=footnotesize,labelfont=bf]{caption}

\usepackage{tikz}
\usetikzlibrary{shapes,intersections}

\usepackage{algorithm}
\usepackage{algpseudocode}
\algrenewcommand\algorithmicrequire{\textbf{Input:}}
\newcommand{\CommentState}[1]{\Statex\hspace{\algorithmicindent}{\color{blue}// #1}}

\DeclareMathOperator*{\argmin}{arg\,min}
\DeclareMathOperator{\prox}{prox}
\DeclareMathOperator{\refl}{refl}
\DeclareMathOperator{\fix}{fix}
\DeclareMathOperator{\diag}{diag}
\DeclareMathOperator{\vect}{vect}

\newtheorem{remark}{Remark}
\newtheorem{proposition}{Proposition}

\newtheorem{assumption}{Assumption}
\newtheorem{lemma}{Lemma}

\newcommand{\norm}[1]{\left\lVert#1\right\rVert}
\newcommand{\expval}[1]{\mathbb{E}\left[#1\right]}

\newcommand{\km}{Krasnosel'ski\u{\i}-Mann\xspace}

\newcommand{\cmark}{\ding{51}}
\newcommand{\xmark}{\ding{55}}
\newcommand{\cvd}{\hfill$\square$}

\def\Top{{\mathcal{T}}}
\def\Tpr{{\mathcal{T}_{\mathrm{PR}}}}

\def\eps{{\mathbold{\epsilon}}}

\def\R{{\mathbb{R}}}
\def\N{{\mathbb{N}}}

\def\x{{\mathbold{x}}}
\def\z{{\mathbold{z}}}
\def\y{{\mathbold{y}}}
\def\w{{\mathbold{w}}}
\def\o{{\mathbold{o}}}
\def\uv{{\mathbold{u}}}
\def\v{{\mathbold{v}}}
\def\g{{\mathbold{g}}}
\def\c{{\mathbold{c}}}

\def\Hm{{\mathbold{H}}}
\def\Dm{{\mathbold{\Delta}}}
\def\A{{\mathbold{A}}}
\def\B{{\mathbold{B}}}
\def\P{{\mathbold{P}}}
\def\I{{\mathbold{I}}}
\def\B{{\mathbold{B}}}
\def\L{{\mathbold{L}}}
\def\T{{\mathbold{T}}}

\def\0{{\mathbf{0}}}
\def\1{{\mathbf{1}}}

\newif\ifarxiv
\arxivtrue % or \arxivtrue

\begin{document}

\title{Asynchronous Distributed Optimization over Lossy Networks via Relaxed ADMM:\\ Stability and Linear Convergence}

\author{Nicola Bastianello, \IEEEmembership{Student Member, IEEE}, Ruggero Carli, \IEEEmembership{Member, IEEE}, Luca Schenato, \IEEEmembership{Fellow, IEEE}, and Marco Todescato, \IEEEmembership{Member, IEEE}
\thanks{N. Bastianello, R. Carli and L. Schenato are with the Department of Information Engineering (DEI), University of Padova, Italy. {\tt\small [bastian4|carlirug|schenato]@dei.unipd.it}.}
\thanks{M. Todescato is with Bosch Center for Artificial Intelligence. Renningen, Germany. {\tt\small mrc.todescato@gmail.com}.}
\thanks{This work has received funding from the Italian Ministry of Education,
University and Research (MIUR) through the PRIN project no. 2017NS9FEY
entitled ``Realtime Control of 5G Wireless Networks: Taming the
Complexity of Future Transmission and Computation Challenges''. The
views and opinions expressed in this work are those of the authors and
do not necessarily reflect those of the funding institution.}
}

\maketitle

\begin{abstract}
In this work we focus on the problem of minimizing the sum of convex cost functions in a distributed fashion over a peer-to-peer network. In particular, we are interested in the case in which communications between nodes are prone to failures and the agents are not synchronized among themselves. We address the problem proposing a modified version of the relaxed ADMM, which corresponds to the Peaceman-Rachford splitting method applied to the dual. By exploiting results from operator theory, we are able to prove the almost sure convergence of the proposed algorithm under general assumptions on the  distribution of communication loss and node activation events. By further assuming the cost functions to be strongly convex, we prove the linear convergence of the algorithm in mean to a neighborhood of the optimal solution, and provide an upper bound to the convergence rate. Finally, we present numerical results testing the proposed method in different scenarios.
\end{abstract}


%--------------------------------------------------------------------------------------------------------------------------
\section{Introduction}\label{sec:intro}

% --------- introduction distributed optimization
From classical control theory to more recent machine learning applications, many problems can be cast as optimization problems \cite{slavakis_modeling_2014} and, in particular, as large-scale optimization problems, given the increasing importance of cyber-physical systems in engineering applications. Stemming from classical optimization theory, in order to break down the computational complexity, parallel and distributed optimization methods have been the focus of a wide branch of research \cite{bertsekas_parallel_1989}. Within this vast topic, typical applications foresee computing nodes to cooperate, through local information exchanges, in order to achieve a desired common goal such as
\begin{equation}\label{eq:generic-distributed-problem}
	\min_{x} \sum_{i=1}^N f_i(x)
\end{equation}
where, usually, each $f_i$ is stored and known by one single node only.

While parallel optimization methods usually rely on a \textit{shared memory} architecture to implement the communication among agents, in distributed systems a \textit{message passing} architecture is employed, in which agents can exchange transmissions with a (subset) of the other agents. The message passing (or \textit{peer-to-peer}) architecture however introduces some issues due to the implementation of the transmission protocols. Indeed, distributed systems may suffer from communication failures, delays, and noise, on top of the possible asynchronism of the agents' activations. In this paper we are interested in solving the distributed problem~\eqref{eq:generic-distributed-problem} in the presence of communication (or packet) losses and asynchronism.

% --------- gradient- and Newton-based methods
A first class of algorithms that has been proposed to solve distributed optimization problems is that of \textit{(sub)gradient-} and \textit{Newton-based} methods.

Distributed gradient descent algorithms in general combine local gradient descent steps with consensus averaging, see for example \cite{nedic_distributed_2009,nedic_convex_2010,nedic_constrained_2010,lobel_distributed_2011,lee_distributed_2013,jakovetic_convergence_2014,nedic_distributed_2015,shi_extra_2015,yuan_convergence_2016,xi_dextra_2017,nedic_achieving_2017,alaviani_distributed_2019}. These algorithms can handle many different scenarios, with smooth and non-smooth costs, over both fixed and time-varying topologies, and over directed and undirected graphs. In general, the convergence of gradient-based methods is sub-linear for convex costs and linear for strongly convex costs. The only method that can handle both packet losses and asynchronous activations of the nodes is \cite{alaviani_distributed_2019}; however, it requires a decreasing step-size and thus, implicitly, that the agents be synchronized.

Newton-based distributed algorithms have been introduced in \cite{wei_distributed_2013_1,wei_distributed_2013_2,varagnolo_newton_2016,mokhtari_network_2017} in synchronous and lossless scenarios. Recently the scheme in \cite{varagnolo_newton_2016} has been extended in \cite{bof_multiagent_2019} to asynchronous and lossy scenarios. However, in~\cite{bof_multiagent_2019}, the convergence is proved only locally and no characterization of the convergence rate is provided.

% --------- ADMM
Other widely studied algorithms for solving distributed optimization problems are the \textit{alternating direction method of multipliers} (ADMM) and the more general \textit{relaxed} ADMM (R-ADMM). This class of algorithms can be defined either as augmented Lagrangian methods \cite{boyd_distributed_2010,eckstein_understanding_2015}, or, within an operator theoretical framework, as the dual of the (relaxed) Peaceman-Rachford splitting \cite{davis_convergence_2016,giselsson_linear_2017}. The latter formulation will be employed in this paper and we refer to \cite{ryu_primer_2016,bauschke_convex_2017} for a background on operator theory and its applications to convex optimization. Typically, the ADMM is derived from a Lagrangian-based formulation, while the R-ADMM is derived in an operator theoretical framework. However, it is known, see \cite{peng_arock_2016}, that the ADMM can be seen as a particular instance of the R-ADMM, obtained setting one of the free parameters equal to a specific value. This slightly reduces the complexity of the updating equations, but the higher flexibility of R-ADMM allows to obtain better convergence properties.

The convergence of ADMM and R-ADMM for convex optimization problems is in general sub-linear, see \textit{e.g.} \cite{boyd_distributed_2010,davis_convergence_2016}, and the same applies for distributed optimization. In an asynchronous scenario, sub-linear convergence can be similarly proved, adopting both the augmented Lagrangian, see \cite{wei_convergence_2013,mota_dadmm_2013,chang_asynchronous_2016_1}, and the splitting operator formulations, see \cite{peng_arock_2016,bianchi_coordinate_2016}. Remarkably, assuming the functional costs are strongly convex, the  distributed implementations of ADMM introduced in \cite{shi_linear_2014,makhdoumi_convergence_2017,iutzeler_explicit_2016}, have been shown to attain linear and global convergence when communications are synchronous and reliable. These results have been extended to asynchronous schemes in \cite{chang_asynchronous_2016_1, chang_asynchronous_2016_2}, though the proposed analysis is limited to master-slave architectures. To the best of our knowledge, \cite{majzoobi_analysis_2018} is the only paper proving convergence of the synchronous ADMM in presence of lossy communications, modeled as i.i.d. binary random variables. However, no characterization of the convergence rate is provided.

The authors of \cite{peng_arock_2016} have derived the R-ADMM within the framework of the ARock algorithm, introduced in the context of parallel computing where agents share a common memory. In \cite{peng_arock_2016} and \cite{hannah_unbounded_2018}, it is shown that the ARock framework successfully handles asynchronous updates and delayed information attaining a sub-linear rate of convergence. However, due to the reliance of the convergence proof on the common memory, ARock it is not suitable to deal with unreliable communications. In \cite{deng_global_2016,giselsson_linear_2017}, the general R-ADMM algorithm is provably shown to be linearly and globally convergent, provided that the dual problem is strongly convex. This result has been extended to randomized scenarios in \cite{bianchi_coordinate_2016,combettes_stochastic_2019}. Unfortunately, strong convexity of the dual problem, is satisfied in the networked optimization scenario of interest only if master-slave architectures are employed, thus preventing the use of fully distributed schemes.

In this paper we present and analyze a modified version of the R-ADMM algorithm which is amenable of distributed implementation in peer-to-peer networks with unreliable communications and asynchronous operations of the agents. The theoretical contribution is twofold:
\begin{itemize}
	\item Deriving the R-ADMM as an application of the Peaceman-Rachford splitting, we are able to exploit recent results on randomized nonexpansive operators to establish the almost sure convergence of the proposed algorithm, provided that mild assumptions on the asynchronous and lossy nature of the network are satisfied.
	\item Further assuming that the local costs are strongly convex and twice differentiable, we show that the convergence is locally linear in mean, and provide an upper bound to the convergence rate.
\end{itemize}

A preliminary version of this paper has appeared in \cite{bastianello_distributed_2018}, where however no asynchronous updates are considered, and no convergence rate analysis is provided.

% --------- paper organization
The remainder of the paper is organized as follows. Section~\ref{sec:background} reviews some concepts in operator theory, and the R-ADMM. Section~\ref{sec:synchronous-ADMM} describes the distributed implementation of R-ADMM to solve \eqref{eq:generic-distributed-problem} and its convergence. Section~\ref{sec:robust-ADMM} analyzes the convergence properties of R-ADMM under asynchronous updates and communication failures. Finally, Section~\ref{sec:simulation} presents some numerical results and Section~\ref{sec:conclusions} concludes the paper.

%--------------------------------------------------------------------------------------------------------------------------
\section{Preliminaries}\label{sec:background}
This Section collects some preliminary definitions in graph theory \cite{mesbahi_graph_2010} and convex analysis \cite{rockafellar_convex_1970}, as well as a brief review of the necessary background regarding operator theory \cite{bauschke_convex_2017} and the R-ADMM \cite{boyd_distributed_2010,giselsson_linear_2017}.

%-------------------------------------------------------------------
\subsection{Notation and useful definitions}
We denote by $\otimes$ the Kronecker product, by $\Lambda(\mathbold{M})$ the spectrum of a matrix $\mathbold{M}$, and by $\operatorname{dist}(x,\mathbb{D}) = \inf_{y \in \mathbb{D}} \norm{x - y}$ the distance between point $x \in \R^n$ and the set $\mathbb{D} \subset \R^n$. $\1_n$ (resp. $\0_n$) denotes the $n$-dimensional vector of all ones (resp. zeros). By $M \succ 0$ we denote that the symmetric matrix $M$ is positive definite.

% --------- graph theory
We denote a graph by $\mathcal{G}=(\mathcal{V},\mathcal{E})$, where $\mathcal{V}$ is the set of $N$ vertices, labeled $1$ through $N$, and $\mathcal{E}$ is the set of undirected edges. For $i \in \mathcal{V}$, by $\mathcal{N}_i$ we denote the set of neighbors of node $i$ in $\mathcal{G}$, namely, $\mathcal{N}_i =\left\{j \in \mathcal{V} \,:\, (i,j) \in \mathcal{E} \right\}$. The degree of each node $i \in \mathcal{V}$ is denoted by $d_i = |\mathcal{N}_i|$. Moreover, in the following we write $M = 2|\mathcal{E}|$, \textit{i.e.}, $M$ counts twice the number of edges in the network.

% --------- convex analysis
Consider the scalar function $f : \R^n \to \R \cup \{+\infty\}$. Then $f$ is said to be \textit{closed} if $\forall a \in \R$ the set $\{x \in \operatorname{dom}(f) \ |\ f(x) \leq a\}$ is closed, and it is \textit{proper} if it does not attain $-\infty$, see \cite{bauschke_convex_2017}. We denote by $\Gamma_0(\R^n)$ the class of convex, closed and proper functions from $\R^n$ to $\R \cup \{+\infty\}$. We define the \textit{convex conjugate} of $f \in \Gamma_0(\R^n)$ as $f^*(w) = \sup_{x \in \R^n} \{ \langle  w, x \rangle - f(x) \}$ for $w \in \R^n$. The convex conjugate belongs to $\Gamma_0(\R^n)$.
Finally, a function $f \in \Gamma_0(\R^n)$ is said to be \textit{$m$-strongly convex}, $m > 0$, if $f - (m/2) \norm{\cdot}^2$ is convex. If $f$ is twice differentiable, $f \in \mathcal{C}^2$, then strong convexity implies that $\nabla^2 f(x) \succ m I$ for all $x \in \R^n$.

%-------------------------------------------------------------------
\subsection{Notions on operator theory}\label{subsec:operator-theory}
By \textit{operator}\footnote{The term \textit{mapping} should actually be used, but in the literature the two are usually employed interchangeably.} on $\R^n$ we mean a map $\Top : \R^n \to \R^n$ that assigns to each point $x$ in $\R^n$ the corresponding point $\Top x \in \R^n$. Given an operator $\Top$, by $\fix(\Top)$ we denote the set of its \textit{fixed points}, that is,  $\fix(\Top) = \{ \bar{x} \in \R^n \ | \ \bar{x} = \Top \bar{x} \}$.

\noindent An operator $\Top$ is \textit{Lipschitz continuous} if there exists $\zeta \geq 0$ such that $\norm{\Top x - \Top y} \leq \zeta \norm{x - y}$ holds for any two $x, y \in \R^n$. In particular, $\Top$ is said to be \textit{nonexpansive} if $\zeta = 1$, and \textit{contractive} if $\zeta \in [0,1)$.
An operator $\Top$ is \textit{averaged} if there exist $\alpha \in (0,1)$ and $\mathcal{R}$ nonexpansive such that we can write $\Top = (1-\alpha) \mathcal{I} + \alpha \mathcal{R}$. Notice that $\fix(\Top) = \fix(\mathcal{R})$.

\noindent An operator $\Top$ is said to be \textit{affine} if there exist $T \in \R^{n \times n}$ and $u \in \R^n$ such that we can write $\Top x = T x + u$, $x \in \R^n$.

Given a function $f \in \Gamma_0(\R^n)$, we define the corresponding \textit{proximal operator} as
$$
	\prox_{\rho f}(x) = \argmin_y \left\{ f(y) + \frac{1}{2\rho} \norm{y - x}^2 \right\},
$$
where $\rho > 0$ is called penalty parameter, and the \textit{reflective operator} as $\refl_{\rho f}(x) = 2\prox_{\rho f}(x) - x$. The proximal is $1/2$-averaged\footnote{This property is also called \textit{firm nonexpansiveness}.} while the reflective is nonexpansive. Observe that the fixed points of $\prox_{\rho f}$ and $\refl_{\rho f}$ coincide with the minimizers of $f$. In general, given $\Top$ nonexpansive, the algorithm for finding its fixed points  is the Krasnosel'skii-Mann (KM) iteration, see \cite{bauschke_convex_2017},
\begin{equation}\label{eq:km-iteration}
	x(k+1) = (1-\alpha) x(k) + \alpha \Top x(k).
\end{equation}

Consider now the convex optimization problem 
\begin{equation}\label{eq:sumFunctions}
	\min_{x \in \R^n} \left\{ f(x) + g(x) \right\}
\end{equation}
with $f, g \in \Gamma_0(\R^n)$. Let us define the \textit{Peaceman-Rachford operator}
$$
	\Tpr = \refl_{\rho g} \circ \refl_{\rho f}
$$
such that the minimizers of the optimization problem are $\prox_{\rho f}(\fix(\Tpr))$. The \km iteration applied to the $\Tpr$ on the auxiliary variable $z$ yields the so called \textit{Peaceman-Rachford splitting} (PRS):
\begin{equation}\label{eq:prs}
	z(k+1) = (1-\alpha) z(k) + \alpha \Tpr z(k), \quad k \in \N
\end{equation}
which is guaranteed to converge to a fixed point of $\Tpr$ if $\alpha \in (0,1)$ and $\rho > 0$, see \cite{bauschke_convex_2017}; a minimizer $\bar{x}$ to \eqref{eq:sumFunctions} is recovered from the limit $\bar{z}$ of the iterate $z(k)$ by computing $\bar{x}=\prox_{\rho g}(\bar{z})$. As show in \cite{bauschke_convex_2017}, the iteration \eqref{eq:prs} can be conveniently implemented by the following updates
\begin{subequations}\label{eq:iterates-splitting}
\begin{align}
	x(k+1) &= \prox_{\rho f}(z(k)) \label{eq:prs-1} \\
	y(k+1) &= \prox_{\rho g}(2x(k+1) - z(k)) \label{eq:prs-2} \\
	z(k+1) &= z(k) + 2\alpha (y(k+1) - x(k+1)) \label{eq:prs-3}
\end{align}
\end{subequations}
where $y$ is an additional auxiliary variable.

%-------------------------------------------------------------------
\subsection{Relaxed ADMM}\label{subsec:R-ADMM}
Consider the following optimization problem
\begin{equation}\label{eq:admm-problem}
\begin{split}
	&\min_{x \in \R^n, y \in \R^m} \left\{ f(x) + g(y) \right\} \\
	&\qquad \text{s.t.}\ A x + B y = c
\end{split}
\end{equation}
where $f \in \Gamma_0(\R^n)$, $g \in \Gamma_0(\R^m)$, $A \in \R^{p \times n}$, $B \in \R^{p \times m}$ and $c \in \R^p$. We assume that \eqref{eq:admm-problem} admits  a finite solution. The dual problem of~\eqref{eq:admm-problem} is (see \cite{peng_arock_2016})
\begin{equation}\label{eq:admm-dual-problem}
	\min_{w \in \R^p} \left\{ d_f(w) + d_g(w) \right\}
\end{equation}
where
$$
	d_f(w) = f^*(A^\top w) \quad \text{and} \quad d_g(w) = g^*(B^\top w) - \langle w, c \rangle.
$$
The \textit{relaxed alternating direction method of multipliers} (R-ADMM) can be derived applying the PRS~\eqref{eq:prs} to solve~\eqref{eq:admm-dual-problem}. In \cite{davis_convergence_2016}, it has been shown that an efficient implementation of \eqref{eq:iterates-splitting} is characterized by the following updates, which involve the primal variables $x$ and $y$, see also \ifarxiv Appendix~\ref{app:derivation-splitting-admm}\else \cite{bastianello_asynchronous_2019}\fi:
\begin{subequations}\label{eq:r-admm}
\begin{align}
	x(k+1) &= \argmin_x \left\{ f(x) - \langle z(k), A x \rangle + \frac{\rho}{2} \norm{A x}^2 \right\} \label{eq:r-admm-x} \\
	w(k+1) &= z(k) - \rho A x(k+1) \label{eq:r-admm-w} \\
	y(k+1) &= \argmin_y \Big\{ g(y) - \langle 2w(k+1)-z(k), B y \rangle + \nonumber \\ &+ \frac{\rho}{2} \norm{By - c}^2 \Big\} \label{eq:r-admm-y} \\
	v(k+1) &= 2w(k+1) - z(k) - \rho (By(k+1) - c) \label{eq:r-admm-v} \\
	z(k+1) &= z(k) + 2\alpha (v(k+1) - w(k+1)). \label{eq:r-admm-z}
\end{align}
\end{subequations}
where \eqref{eq:r-admm-x}, \eqref{eq:r-admm-w} implements \eqref{eq:prs-1}, while \eqref{eq:r-admm-y}, \eqref{eq:r-admm-v} implements \eqref{eq:prs-2}.
The convergence of the PRS guarantees, in turn, the convergence of $\{ x(k) \}_{k \in \N}$ and $\{ y(k) \}_{k \in \N}$ to an optimal solution of the primal~\eqref{eq:admm-problem}. Indeed problem \eqref{eq:admm-problem} is convex with linear constraints and strong duality holds.

The R-ADMM is a generalized version of the classical ADMM described \textit{e.g.} in \cite{boyd_distributed_2010}; indeed it is possible to see that when $\alpha = 1/2$ the former recovers the latter, see Remark \ref{rem:Lagrangian-based} in Section \ref{sec:synchronous-ADMM}.

%--------------------------------------------------------------------------------------------------------------------------
\section{R-ADMM for Distributed Optimization}\label{sec:synchronous-ADMM}
In this Section we formulate the distributed optimization problem of interest and we show how the R-ADMM is suited to solve it.

%-------------------------------------------------------------------
\subsection{Problem formulation}
Consider the undirected, connected graph $\mathcal{G}=(\mathcal{V},\mathcal{E})$ with $N$ nodes. We are interested in solving 
\begin{equation}\label{eq:distributed-problem}
	\min_{x \in \R^n} \sum_{i=1}^N f_i(x)
\end{equation}
over the network $\mathcal{G}$ where the cost function $f_i \in \Gamma_0(\R^n)$ is known only to the $i$-th node, and nodes can communicate only with their neighbors. We assume that \eqref{eq:distributed-problem} admits at least one finite solution. 

In order to apply the R-ADMM to problem \eqref{eq:distributed-problem} we reformulate it as follows. First, a local copy $x_i \in \R^n$ of the decision variable $x$ is assigned to each agent. Therefore, as long as $\mathcal{G}$ is connected, problem~\eqref{eq:distributed-problem} is equivalent to
\begin{equation}\label{eq:distributed-primal}
\begin{split}
	\min_{x_i, i \in \mathcal{V}} & \sum_{i=1}^N f_i(x_i) \\
	\text{s.t.} \ \ & x_i = x_j \quad \forall (i,j) \in \mathcal{E}.
\end{split}
\end{equation}
Indeed the consensus constraints $x_i = x_j$ impose that any optimal solution of~\eqref{eq:distributed-primal} satisfies $x_1 = \ldots = x_N = \bar{x}$, with $\bar{x}$ a solution to~\eqref{eq:distributed-problem}. Introducing the \textit{bridge variables} $y_{ij}$ and $y_{ji}$ for each edge $(i,j) \in \mathcal{E}$, the consensus constraints can be equivalently rewritten as
\begin{equation}\label{eq:BridgeVariables}
	x_i = y_{ij}, \quad x_j = y_{ji} \quad \text{and} \quad y_{ij} = y_{ji} \quad \forall (i,j) \in \mathcal{E}.
\end{equation}
Defining the vectors $\x = [x_1^\top, \ldots, x_N^\top]^\top \in \R^{nN}$ and $\y = [\ldots, \{ y_{ij}^\top \}_{j \in \mathcal{N}_i}, \ldots]^\top \in \R^{nM}$, the constraints in \eqref{eq:BridgeVariables} can be compactly rewritten as\footnote{Hereafter, boldface letters will denote vectors and matrices built stacking local quantities.}
$$
	\A \x + \B \y = \0 \quad \text{and} \quad \y = \P \y
$$
where 
$$
	\A = \begin{bmatrix}
		\1_{d_1} & \0_{d_1} & \0_{d_1} & \cdots & \0_{d_1} \\
		\0_{d_2} & \1_{d_2} & \0_{d_2} & \cdots & \0_{d_2} \\
		& & \ddots & & \\
		& & & \ddots & \\
		\0_{d_N} & \cdots & \cdots & \0_{d_N} & \1_{d_N}
	\end{bmatrix} \otimes I_n \in \R^{nM \times nN},
$$
$\B = -\I_{nM}$, $\P$ is a permutation matrix that swaps $y_{ij}$ with $y_{ji}$. We remark that $\A$ in general is not full row rank.

Finally, define $f(\x) = \sum_{i=1}^N f_i(x_i)$ and $g(\y) = \iota_{(\I - \P)}(\y)$, where the indicator function $\iota_{(\I - \P)}(\y)$ is equal to $0$ if $(\I - \P) \y = \0$ and $+\infty$ otherwise. Hence problem \eqref{eq:distributed-primal}  can be equivalently formulated as
\begin{align}\label{eq:primal-indicator-f}
\begin{split}
	& \min_{\x,\y}\left\{f(\x)+\iota_{(\I-\P)}(\y)\right\}\\
	& \ \text{s.t.}\ \ \A\x-\y=0.
\end{split}
\end{align}
Problem~\eqref{eq:primal-indicator-f} is in the form of~\eqref{eq:admm-problem} and thus we can apply the R-ADMM algorithm to solve it.

%-------------------------------------------------------------------
\subsection{Distributed R-ADMM}
The particular separable structure of the functions $f(\x)$, $g(\y)$, and of matrices $\A$, $\B$, allows us to derive simplified equations for the R-ADMM algorithm that involve only the update of the $\x$ and $\z$ variables, and that are amenable of distributed implementations.

\noindent Indeed, it can be shown that the equations~\eqref{eq:r-admm} applied to problem \eqref{eq:primal-indicator-f} reduce to
\begin{subequations}\label{eq:distributed-admm}
\begin{align}
	x_i(k+1) &= \argmin_{x_i} \bigg\{ f_i(x_i) + \nonumber \\ & \qquad\quad - \langle \sum_{j \in \mathcal{N}_i} z_{ij}(k), x_i \rangle +  \frac{\rho d_i}{2} \norm{x_i}^2 \bigg\} \label{eq:distributed-admm-x} \\
	z_{ij}(k+1) &= (1-\alpha) z_{ij}(k) - \alpha z_{ji}(k) + 2\alpha\rho x_j(k+1) \label{eq:distributed-admm-z}
\end{align}
\end{subequations}
for all $i \in \mathcal{V}$ and $j \in \mathcal{N}_i$. See Appendix~\ref{app:derivation-distributed-admm} for the derivation. Observe that, since $\B=-\I$, the dimension of $\z$ is equal to the dimension of $\y$, \textit{i.e.}, for all $(i,j) \in \mathcal{E}$ there are the variables $z_{ij}$ and $z_{ji}$. Interestingly, one can see that \eqref{eq:distributed-admm-x}  can be rewritten as
\begin{equation}\label{eq:primal-update}
	x_i(k+1) = \prox_{f_i / (\rho d_i)}\left( [\A^\top \z(k)]_i / (\rho d_i) \right),
\end{equation}
where $[\A^\top \z(k)]_i$ denotes the $i$-th component of the vector $[\A^\top \z(k)]$, see Lemma~\ref{lem:x_i-prox} in Appendix~\ref{app:linear_convergence}. A straightforward implementation of~\eqref{eq:distributed-admm} has the node $i$ storing and updating $x_i$ and $\{ z_{ij} \}_{j \in \mathcal{N}_i}$. Notice that, while~\eqref{eq:distributed-admm-x} can be computed using only local information, \textit{i.e.}, the local cost $f_i$ and $\{ z_{ij} \}_{j \in \mathcal{N}_i}$, update~\eqref{eq:distributed-admm-z} requires communication with $i$'s neighbors, that is, transmission of $z_{ji}(k)$ and $x_j$ from node $j \in \mathcal{N}_i$. In particular we assume that node $j$ sends to node $i$ the packet
$$
	q_{j \to i} = -z_{ji}(k) + 2\rho x_j(k+1)
$$
and, consequently, node $i$ performs the update
\begin{equation}\label{eq:distributed-admm-z-with-q}
	z_{ij}(k+1) = (1-\alpha) z_{ij}(k) + \alpha q_{j \to i}.
\end{equation}
Algorithm~\ref{alg:distributed-admm} describes the implementation of the distributed R-ADMM.

\begin{algorithm}[!ht]
\caption{Distributed R-ADMM.}
\label{alg:distributed-admm}
\begin{algorithmic}[1]
	\Require For each node $i$, initialize $x_i(0)$ and $z_{ij}(0)$, $j \in \mathcal{N}_i$.
	\For{$k = 0,1,\ldots$ every agent $i$}
	\CommentState{local update}
	\State compute $x_i(k+1)$ according to \eqref{eq:distributed-admm-x}
	\CommentState{transmission}
	\For{each neighbor $j \in \mathcal{N}_i$}
	\State compute and transmit the temporary variable $q_{i \to j}$
	\EndFor
	\State gather $q_{j \to i}$ from each neighbor $j$
	\CommentState{auxiliary update}
	\State compute $z_{ij}(k+1)$ according to~\eqref{eq:distributed-admm-z-with-q}
	\EndFor
\end{algorithmic}
\end{algorithm}

The following convergence result is a direct consequence of the convergence of the Peaceman-Rachford splitting, proved \textit{e.g.} in \cite[Th.~26.11]{bauschke_convex_2017}.

\begin{proposition}\label{cr:convergence}
Consider problem~\eqref{eq:distributed-problem} with $f_i \in \Gamma_0(\R^n)$, and let $ 0 < \alpha < 1$ and $\rho > 0$. Then, for any initial condition $\z(0) \in \R^{nM}$, the trajectories $k \mapsto x_i(k)$, $i \in \mathcal{V}$, generated by Algorithm~\ref{alg:distributed-admm}, converge to an optimal solution $\bar{x}$ of \eqref{eq:distributed-problem}, \textit{i.e.},
$$
	\lim_{k \to \infty} x_i(k) = \bar{x}, \qquad \forall i \in \mathcal{V}.
$$
%for any $x_i(0)$ and $z_{ij}(0)$, $j \in \mathcal{N}_i$.
\end{proposition}

\begin{remark}
Notice that the statement of Proposition~\ref{cr:convergence} considers only the initial condition of variable $\z$ and not of $\x$. The reason is related to update \eqref{eq:distributed-admm-x} where it is clear that $\x(1)$ depends only on $\z(0)$ and not on $\x(0)$.
\end{remark}

\begin{remark}[\bf Comparison with ARock {\cite{peng_arock_2016}}]\label{rem:ARock}
The formulation of the R-ADMM presented in Algorithm~\ref{alg:distributed-admm} is derived using the same idea employed in \cite{peng_arock_2016} of interpreting the R-ADMM as an application of the PRS to the dual problem. Also the R-ADMM algorithm proposed in \cite{peng_arock_2016} to solve problem \eqref{eq:generic-distributed-problem} (see section 2.6.2), involves only the use of variables $x_i$, $z_{ij}$ but the actual implementation differs from Algorithm~\ref{alg:distributed-admm}.
Additionally, it is worth mentioning that the authors of \cite{peng_arock_2016} have derived the R-ADMM within the framework of the ARock algorithm, introduced in the context of parallel computing where agents share a common memory. In \cite{peng_arock_2016} and \cite{hannah_unbounded_2018} it is shown that the ARock framework successfully handles asynchronous updates and (possibly unbounded) delayed information. However, due to the reliance of the convergence proof on the common memory, ARock is not suitable to deal with the lossy and asynchronous framework of interest in this paper.
\end{remark}

%-------------------------------------------------------------------
\subsection{Linear local convergence for strongly convex costs}\label{subsec:linear-convergence-synchronous}
In this section we prove the local linear convergence of the distributed R-ADMM under the assumption that the local costs are strongly convex and twice continuously differentiable. Notice that, under these assumptions there exists a unique minimizer $x^*$ for problem \eqref{eq:distributed-problem}. 

It is worth stressing that for particular distributed and centralized formulations of the R-ADMM (see \cite{shi_linear_2014, makhdoumi_convergence_2017, giselsson_linear_2017}), especially of the classical ADMM, it is actually possible to prove global linear convergence under milder assumptions than the ones made in this section. However the results in \cite{shi_linear_2014, makhdoumi_convergence_2017, giselsson_linear_2017} can be applied only partially to the scenario of our interest and we refer to  Remark~\ref{rem:Conv-rate} for a detailed discussion.

The idea behind the result of Proposition~\ref{pr:local-linear-convergence} below is that, in a neighborhood of the optimal solution, the strong convexity and double continuous differentiability of the local costs allow us to rewrite the Peaceman-Rachford splitting applied to the dual of the distributed problem as a \textit{perturbed affine operator}. Indeed, it is possible to write the update of the auxiliary variables in compact form as
\begin{equation}\label{eq:affine-perturbed-z}
	\z(k+1) = \T \z(k) + \uv + \o'(\x(k+1) - \x^*)
\end{equation}
where $\x^* = \1_N \otimes x^*$, and $\T \in \R^{nM \times nM}$ is such that
\begin{equation}\label{eq:T}
	\T = (1-\alpha) \I - \alpha \P + 2\alpha\rho \P \A \Hm^{-1} \A^\top,
\end{equation}
with $\Hm = \operatorname{blk\,diag} \left\{ \rho d_i I_n + \nabla^2 f_i(x^*) \right\}$, $\uv \in \R^{nM}$ is a constant vector depending on the gradient and Hessian of $f(\x)$ evaluated at $\x^*$, and $\o' : \R^{nN} \to \R^{nM}$ is a vanishing function for $\x$ approaching the optimum, that is,
$$
	\norm{\o'(\x(k+1) - \x^*)} / \norm{\x(k+1) - \x^*} \to 0
$$ 
as $\x(k+1) \to \x^*$. All the details can be found in Appendix~\ref{app:linear_convergence}.
 
In Lemma \ref{lem:T_properties} in Appendix~\ref{app:linear_convergence}, it is established that the eigenvalues of $\T$ are either equal to $1$ or strictly inside the unitary circle, with the eigenvalues equal to $1$ all being semi-simple. The largest (in absolute value) eigenvalue smaller than $1$ of $\T$ is an upper bound to the convergence rate of the $x_i$'s trajectories toward the optimum. This fact is formally stated in the following Proposition.

\begin{proposition}\label{pr:local-linear-convergence}
Assume that the local costs $f_i$ are strongly convex and twice continuously differentiable. Then there exists $\epsilon > 0$ such that, if $\operatorname{dist}(\z(0), \fix(\Tpr)) \leq \epsilon$, then Algorithm~\ref{alg:distributed-admm} converges linearly fast, \textit{i.e.}
$$
	\norm{x_i(k) - x^*} \leq C \gamma^k \qquad i \in \mathcal{V},
$$
with $C > 0$, and $0 \leq \gamma \leq \gamma_{\mathrm{M}} < 1$, where $\gamma_{\mathrm{M}}$ is the largest eigenvalue of $\T$ different from one, \textit{i.e.}
$$
	\gamma_{\mathrm{M}} = \max \left\{ |\lambda| \ \mathrm{s.t.} \ \lambda \in \Lambda(\T), \ |\lambda| < 1 \right\}.
$$
\end{proposition}
\begin{proof}
See~Appendix~\ref{app:linear_convergence}.
\end{proof}

While we refer to Appendix~\ref{app:linear_convergence} for the proof of this result, hereafter we comment some interesting details. These results will play an important role also in the analysis performed in Section~\ref{sec:robust-ADMM} in the presence of asynchronous updates and lossy communications, which is the main novelty of this paper.

\noindent The proof of Proposition \ref{pr:local-linear-convergence} relies on the following two facts:
\begin{enumerate}
	\item[\textit{(i)}] the set $\fix(\Tpr)$ is an affine space such that, given any two fixed points $\bar{\z}, \bar{\z}' \in \fix(\Tpr)$, it holds $\bar{\z}' - \bar{\z} \in \ker(\A^\top)$;
	\item[\textit{(ii)}] exploiting \eqref{eq:primal-update} and the fact that $\prox_{f_i / (\rho d_i)}$, $i \in \mathcal{V}$, is contractive, it is possible to upper bound the primal error $\x(k+1) - \x^*$ with the auxiliary error $\z(k) - \bar{\z}$, for any $\bar{\z} \in \fix(\Tpr)$. Specifically, we have that
	\begin{equation}\label{eq:Bound_primal}
		\norm{\x(k+1) - \x^*} \leq \zeta \norm{\A^\top (\z(k) - \bar{\z})},
	\end{equation}
	where $\zeta \in (0,1)$ is a suitable constant depending on the curvature of the $f_i$'s and on the topology of $\mathcal{G}$; for more details see \eqref{eq:primal-error-bound} and \eqref{eq:zeta} in Appendix \ref{app:linear_convergence}.
\end{enumerate}
Since $\z(k)$ converges to a fixed point then, from \eqref{eq:Bound_primal} and fact \textit{(ii)}, we have $\x(k) \to \x^*$. As observed in Appendix~\ref{app:linear_convergence}, if all the $f_i$ were quadratic functions, then \eqref{eq:affine-perturbed-z} would reduce to the linear update $\z(k+1) = \T \z(k) + \uv$. The convergence would thus be global, linear and with rate upper bounded by $\gamma_{\mathrm{M}}$.

\noindent In the proof of Proposition \ref{pr:local-linear-convergence} we show that this linear convergence is not deteriorated by the presence of the nonlinear term  $\o'(\x(k+1) - \x^*)$, though the price to be paid is that the convergence is not guaranteed to be global but only local.

Some remarks are now in order to better cast Algorithm \ref{alg:distributed-admm} within the existing literature. In particular Remarks \ref{rem:Lagrangian-based} and \ref{rem:Node-Edge} discuss the connection with Augmented Lagrangian-based and node-/edge-based formulations, respectively. Remark \ref{rem:Conv-rate} provides a further discussion on the convergence rate.

\begin{remark}[\bf Lagrangian based R-ADMM]\label{rem:Lagrangian-based}
The R-ADMM described in section~\ref{subsec:R-ADMM}  can be interpreted in the framework of \textit{augmented Lagrangian methods}. Indeed, define the augmented Lagrangian
\begin{equation}\label{eq:augmented-lagrangian}
\begin{split}
	\mathcal{L}_\rho(\x,\y,\w) &= f(\x) + g(\y) - \langle \w, \A \x + \B \y - \c \rangle + \\ &+ \frac{\rho}{2} \norm{\A \x + \B \y - \c}^2
\end{split}
\end{equation}
where $\w$ is the Lagrange multipliers' vector.
Then, the R-ADMM in~\eqref{eq:r-admm} is equivalent to the following updates, \ifarxiv see \cite{davis_convergence_2016} and Appendix~\ref{app:equivalence}\else \cite{davis_convergence_2016,bastianello_asynchronous_2019}\fi:
\begin{subequations}\label{eq:lagrangian-r-admm}
\begin{align}
	\x(k+1) &= \argmin_\x \Big\{ \mathcal{L}_\rho(\x,\y(k),\w(k)) + \label{eq:lagrangian-r-admm-x} \\ &+ \rho (2\alpha-1) \langle \A \x(k) + \B \y(k) - \c, \A \x \rangle \Big\} \nonumber \\
	\w(k+1) &= \w(k) - \rho (2\alpha-1) (\A \x(k) + \B \y(k) - \c) + \nonumber \\ &- \rho (\A \x(k+1) + \B \y(k) - \c) \label{eq:lagrangian-r-admm-w} \\
	\y(k+1) &= \argmin_\y \mathcal{L}_\rho(\x(k+1),\y,\w(k+1)). \label{eq:lagrangian-r-admm-y}
\end{align}
\end{subequations}
In particular, given $\x(0), \y(0), \w(0)$, if $\z(0) = \w(0) - \rho(2\alpha-1) (\A \x(0) + \B \y(0) - \c) - \rho (\B \y(0) - \c)$, then the $\x$ and $\y$ trajectories generated by~\eqref{eq:r-admm} and~\eqref{eq:lagrangian-r-admm} coincide, \ifarxiv see Appendix \ref{app:equivalence}\else see \cite{bastianello_asynchronous_2019}\fi. Observe that, if $\alpha=1/2$, we recover the classical ADMM described \textit{e.g.} in \cite{boyd_distributed_2010}. The choice of analyzing the more general R-ADMM relies on the fact that, by properly tuning the parameter $\alpha$, we can achieve better performance than the classical ADMM  as observed \textit{e.g.} in \cite{ghadimi_optimal_2015}, proved in \cite{giselsson_linear_2017}, and evidenced by the numerical results in Section~\ref{sec:simulation}. Interestingly, also the augmented Lagrangian-based R-ADMM in \eqref{eq:lagrangian-r-admm} is amenable of a \textit{distributed} implementation when applied to \eqref{eq:primal-indicator-f}, which is described by the following updates
\begin{subequations}\label{eq:lagrangian-r-admm-distributed}
\begin{align}
	x_i(k+1) &= \argmin_{x_i} \bigg\{ f_i(x_i) + \frac{\rho d_i}{2} \norm{x_i}^2 + \\ &\hspace{-0.85cm} - \langle x_i, \sum_{j \in \mathcal{N}_i} [w_{ij}(k) - 2\alpha\rho y_{ij}(k) - \rho(2\alpha-1) x_i(k)] \rangle \bigg\}\nonumber \\
	w_{ij}(k+1) &= w_{ij}(k) - \rho(2\alpha-1) (x_i(k) - y_{ij}(k)) + \\ &\qquad- \rho (x_i(k+1) - y_{ij}(k)) \nonumber\\
	y_{ij}(k+1) &= \frac{1}{2\rho} \big[ x_i(k+1) + x_j(k+1) + \\ &\qquad- w_{ij}(k+1) - w_{ji}(k+1) \big].\nonumber
\end{align}
\end{subequations}
Clearly, this formulation requires each node $i \in \mathcal{V}$ to store the variables $x_i$, $\{ y_{ij} \}_{j \in \mathcal{N}_i}$ and $\{ w_{ij} \}_{j \in \mathcal{N}_i}$, and to update them exchanging information only with its neighbors. Therefore in terms of storage requirement Algorithm~\ref{alg:distributed-admm} is better than the augmented Lagrangian formulation, as evidenced by Table~\ref{tab:operator-vs-lagrangian-admm}.

\begin{table}[!ht]
\centering
\caption{Comparison of R-ADMM formulations.}
\label{tab:operator-vs-lagrangian-admm}
\begin{tabular}{ccc}
	\hline 
	& Operator R-ADMM & Lagrangian R-ADMM \\ 
	\hline 
	Memory & $d_i + 1$ & $2d_i + 1$ \\
	Transmit & $d_i$ & $2d_i + 1$ \\
	\hline
\end{tabular}
\end{table}
\end{remark}

\begin{remark}[\bf Node- and edge-based ADMM]\label{rem:Node-Edge}
In Algorithm~\ref{alg:distributed-admm} and in the corresponding Lagrangian-based formulation, the number of variables that each node stores scales with $d_i$, see Table~\ref{tab:operator-vs-lagrangian-admm}. This is due to the fact that each node stores an auxiliary variable for each of the edges it is part of -- hence the name \textit{edge-based} -- incurring in a worst case memory requirement of $O(N)$. A different formulation of the R-ADMM, called \textit{node-based}, can be given, in order to guarantee that the local storage requirement is constant, \textit{i.e.} $O(1)$. Node-based formulations of the classical ADMM are employed \textit{e.g.} in \cite{mota_dadmm_2013,shi_linear_2014,makhdoumi_convergence_2017}. Notice that Algorithm~\ref{alg:distributed-admm} can be reformulated as a node-based method if each node $i \in \mathcal{V}$ stores and updates the variables $z_i' = \sum_{j \in \mathcal{N}_i} z_{ij}$ and $z_i'' = \sum_{j \in \mathcal{N}_i} z_{ji}$ instead of the $\{ z_{ij} \}_{j \in \mathcal{N}_i}$ auxiliary variables.

However, as observed in \cite{majzoobi_analysis_2018}, in general node-based ADMM formulations \textit{are not robust to packet losses}. Indeed, the convergence of a node-based ADMM is guaranteed only if, at each iteration $k$, the graph resulting from the removal of faulty edges is still connected. Edge-based formulations are instead necessary in order to remove this (rather demanding) assumption, such as the one proposed in \cite{majzoobi_analysis_2018} to handle (uniformly distributed) packet losses, and Algorithm~\ref{alg:robust-distributed-admm} proposed in this paper. Intuitively, the use of bridge variables is necessary in order to keep track of the packets received at any given time from each of the neighbors.
\end{remark}

\begin{remark}[\bf Further discussion on the convergence rate]\label{rem:Conv-rate}
In recent years there has been an increasing interest in characterizing the convergence rate of both centralized and distributed implementations of R-ADMM and classical ADMM. A special effort has been devoted to provide conditions under which the convergence is guaranteed to be linear. Next, it is worth summarizing some of the main results comparing them to Proposition \ref{pr:local-linear-convergence}. Interestingly, the distributed implementations of the standard ADMM introduced in  \cite{shi_linear_2014} and \cite{makhdoumi_convergence_2017} have been shown to attain global and linear convergence provided that the Lagrange multipliers satisfy some particular initializations and under the assumptions that the local costs are strongly convex and with Lipschitz continuous gradient. Though these assumptions are milder than the one made in Proposition \ref{pr:local-linear-convergence}, the results in \cite{shi_linear_2014} and \cite{makhdoumi_convergence_2017},  when interpreted in the context of the more general R-ADMM, are valid only for the case $\alpha = 1/2$, while the result in section~\ref{subsec:linear-convergence-synchronous} holds true for any $\alpha$ within the interval $(0,1)$. Moreover, for the case $\alpha = 1/2$, the analysis performed in \cite{shi_linear_2014} could be mimicked for the distributed Lagrangian-based implementation described in \eqref{eq:lagrangian-r-admm-distributed}, thus obtaining a global linear rate when $\alpha=1/2$ also for the algorithm proposed in this paper.

Concerning the general Peaceman-Rachford splitting applied to \eqref{eq:admm-problem}, the authors of \cite{giselsson_linear_2017} have shown that the R-ADMM algorithm converges linearly to an optimal solution provided that the matrix $\A$ is full row rank, which guarantees that the Peaceman-Rachford operator is contractive. Under the same assumption, the linear convergence extends also to randomized updates \cite{combettes_stochastic_2019}. However, in the distributed optimization scenario of interest \textit{$\A$ is not full row rank}, since $M>N$ and therefore the aforementioned convergence results cannot be applied.  In particular, the loss of row rank for $\A$ implies that the dual function $d_f$ is only convex \textit{but not strongly convex}, and, hence that the Peaceman-Rachford operator is only nonexpansive.
A notable exception arises when adopting master-slave architectures, which are characterized by a master node connected to $N$ other nodes (the slaves). Indeed, in this setting only one bridge variable is introduced for any edge, and thus $\A = \I_{nN}$, which is full row rank. The implementation of the ADMM in this setup envisions the slave nodes performing local updates of the primal variables, and the master node updating the $N$ dual variables.
\end{remark}

%--------------------------------------------------------------------------------------------------------------------------
\section{Asynchronous Distributed R-ADMM over Lossy Networks}\label{sec:robust-ADMM}
Algorithm \ref{alg:distributed-admm} works under the standing assumption that the communication channels are reliable and that the nodes update synchronously. The goal of this Section is to relax these requirements and to show how Algorithm \ref{alg:distributed-admm} can be modified to still guarantee convergence, under probabilistic assumptions on communication failures and asynchronous updates, and to characterize its linear convergence in mean.

%-------------------------------------------------------------------
\subsection{Robust and Asynchronous R-ADMM}
Consider Algorithm~\ref{alg:distributed-admm}, and notice that node $i$ at iteration $k$ receives the packet $q_{j \to i}$ from $j \in \mathcal{N}_i$ only if the two following conditions are satisfied: \textit{(i)} node $j$ performs an update of $x_j$ at iteration $k$; and \textit{(ii)} the packet $q_{j \to i}$ is not lost.

\noindent Now, for any $k=0, 1, \ldots$, let us define the set of random variables $\{ \mu_i (k)\}_{i \in \mathcal{V}}$, such that the realization of $\mu_i(k)$ is $1$ if node $i$ performs an update during the $k$-th iteration, $0$ otherwise. Similarly, provided that $\mu_i(k)=1$, we define the set of variables $\{ \lambda_{i \to j}(k) \}_{i \in \mathcal{V}, j \in \mathcal{N}_i}$ such that the realization of $\lambda_{i \to j}(k)$ is $0$ if $q_{i \to j}$ is delivered to $j$, $1$ otherwise. Within this formalism, we see that node $i$ can carry out an update of $z_{ij}$ at iteration $k$ provided that $\mu_j(k) = 1$ and $\lambda_{j \to i}(k) = 0$. To simplify the theoretical analysis, we define the set of random variables $\{ \beta_{ij} \}_{i \in \mathcal{V}, j \in \mathcal{N}_i}$ such that
$$
	\beta_{ij}(k) = \begin{cases}
		1 & \text{if} \ \mu_j(k) = 1 \ \text{and} \ \lambda_{j \to i}(k) = 0, \\
		0 & \text{otherwise}.
	\end{cases}
$$
We make the following probabilistic Assumption on the variables $\beta_{ij}$.

\begin{assumption}\label{as:asynchronous-lossy-scenario}
The random variables $\{ \beta_{ij}(k)\,: \,\, i \in \mathcal{V}, \,j \in \mathcal{N}_i, \,k \in \N\}$ are mutually independent over $k$, namely, given $\beta_{ij}(k)$ and $\beta_{hl}(\ell)$ for any $(i,j), (h,l) \in \mathcal{E}$, they are independent if $k \neq \ell$. Moreover, there exists a $M$-uple $\{p_{ij}\,:\, i \in \mathcal{V}, \,j \in \mathcal{N}_i, \,0< p_{ij}<1 \}$ such that
\begin{equation}\label{eq:probabilityBeta}
	\mathbb{P}[\beta_{ij}(k) = 1] = p_{ij},
\end{equation}
for all $k \in \N$.
\end{assumption}

Observe that Assumption~\ref{as:asynchronous-lossy-scenario} requires only independence over time, but not among the random variables at the same iteration $k$. Moreover, as consequence of \eqref{eq:probabilityBeta}, each variable $z_{ij}$ has a nonzero probability of being updated at each iteration $k$.  Assumption~\ref{as:asynchronous-lossy-scenario} could have been stated equivalently in terms of $\mu_i$ and $\lambda_{i \to j}$, assuming nonzero probabilities for the occurrence of update and packet delivery events, and mutual independence over time.  

\begin{remark}[\bf Uniform probabilities]\label{rem:uniform-rv}
Assume the random variables $\{ \mu_i(k) \}_{i \in \mathcal{V}}$ and $\{ \lambda_{i \to j}(k) \}_{i \in \mathcal{V}, j \in \mathcal{N}_i}$ are i.i.d., such that $\expval{\mu_i(k)} = p_\mu$ and $\expval{\lambda_{i \to j}(k)} = p_\lambda$, for all $i \in \mathcal{V}$, $j \in \mathcal{N}_i$, and $k \in \N$. Then $\{ \beta_{ij}(k) \}_{i \in \mathcal{V}, j \in \mathcal{N}_i}$ are uniformly distributed with probability $p_\beta = p_\mu (1-p_\lambda)$, but in general are not independent, since $\{ \beta_{ij}(k) \}_{j \in \mathcal{N}_i}$ all depend on $\mu_i(k)$.
\end{remark}

In Algorithm~\ref{alg:robust-distributed-admm} we describe the modified version of Algorithm~\ref{alg:distributed-admm} that can handle asynchronous updates and packet losses. If node $i$ at iteration $k$ is selected, then it updates $x_i$ and computes the variables $q_{i \to j}$, $j \in \mathcal{N}_i$, transmitting them to its neighbors. If node $j$ receives $q_{i \to j}$ then it updates the variable $z_{ji}$, otherwise it leaves it unchanged.

\begin{algorithm}[!ht]
\caption{Robust and asynchronous distributed R-ADMM.}
\label{alg:robust-distributed-admm}
\begin{algorithmic}[1]
	\Require For each node $i$, initialize $x_i(0)$ and $z_{ij}(0)$, $j \in \mathcal{N}_i$.
	\For{$k = 0,1,\ldots$ each agent $i$}
	\CommentState{local update and transmission}
	\If{scheduled to update}
	\State compute $x_i(k+1)$ according to \eqref{eq:distributed-admm-x}
	\For{each neighbor $j \in \mathcal{N}_i$}
	\State compute and transmit $q_{i \to j}$
	\EndFor
	\EndIf
	\CommentState{auxiliary update}
	\For{each $j \in \mathcal{N}_i$}
	\State if $q_{j \to i}$ was received, compute $z_{ij}(k+1)$ according\\\hspace{2.9em}to~\eqref{eq:distributed-admm-z-with-q}
	\EndFor
	\EndFor
\end{algorithmic}
\end{algorithm}

Notice that node $i$ updates the variable $z_{ij}$ only if it receives the packet $q_{j \to i}$.
Making use of the random variables $\beta_{ij}$, we can thus describe the update step for the auxiliary variables in the following compact form
\begin{equation}\label{eq:random-z_ij}
\begin{split}
	z_{ij}(k+1) &= (1-\beta_{ij}(k)) z_{ij}(k) + \\ &\qquad + \beta_{ij}(k) \left( (1-\alpha) z_{ij}(k) + \alpha q_{j \to i} \right).
\end{split}
\end{equation}

The following convergence result holds as a consequence of the convergence of the Peaceman-Rachford splitting with random coordinate updates, see \cite{combettes_stochastic_2015,bianchi_coordinate_2016}.

\begin{proposition}\label{cr:convergence_lossy}
Consider problem~\eqref{eq:distributed-problem} with $f_i \in \Gamma_0(\R^n)$. Assume Assumption~\ref{as:asynchronous-lossy-scenario} holds, and let $0 < \alpha < 1$ and $\rho > 0$. Then for any initial condition $\z(0) \in \R^{nM}$, the trajectories $k \mapsto x_i(k)$, $i \in \mathcal{V}$, generated by Algorithm~\ref{alg:robust-distributed-admm} converge almost surely to an optimal solution $\bar{x}$ of~\eqref{eq:distributed-problem}, that is,
$$
	\mathbb{P}\left[ \lim_{k \to \infty} x_i(k) = \bar{x} \right] = 1, \qquad \forall i \in \mathcal{V}.
$$
\end{proposition}
\begin{proof}
See Appendix~\ref{app:convergence-lossy}.
\end{proof}

%-------------------------------------------------------------------
\subsection{Mean linear convergence for strongly convex costs}

\begin{table*}[!ht]
\centering
\caption{Comparison of convergence results for different distributed ADMM formulations.}
\label{tab:convergence-results-comparison}
\begin{tabular}{ccccccc}
\hline
 	& \small Reference & \small Formulation & \small $\alpha$ & \small Linear convergence & \small Asynchronous updates & \small Packet loss \\ 
\hline 
\multirow{5}{*}{\small \shortstack{Augmented\\Lagrangian\\ADMM}}
	& Shi \textit{et al.} \cite{shi_linear_2014} & node-based & $1/2$ & global & \xmark & \xmark \\
	& Makhdoumi \& Ozdaglar \cite{makhdoumi_convergence_2017} & node-based & $1/2$ & global & \xmark & \xmark \\ 
	& Majzoobi \textit{et al.} \cite{majzoobi_analysis_2018} & node-, edge-based & $1/2$ & \xmark & \xmark & \cmark (uniform distr.) \\
	& Chang \textit{et al.} \cite{chang_asynchronous_2016_2} & master-slave$^\dagger$ & $1/2$ & global & \cmark (for master) & \xmark \\
	& Iutzeler \textit{et al.} \cite{iutzeler_explicit_2016} & clustered & $1/2$ & local & \cmark (for clusters) & \xmark \\
\hline 
\multirow{5}{*}{\small \shortstack{Splitting\\ADMM}}
	& Bianchi \textit{et al.} \cite{bianchi_coordinate_2016} & edge-based & $1/2$ & \xmark & \cmark & \xmark \\
	& Giselsson \& Boyd \cite{giselsson_linear_2017} & master-slave$^\dagger$ & $(0,1)$ & global & \xmark & \xmark \\
	& Combettes \& Pesquet \cite{combettes_stochastic_2019} & master-slave$^\dagger$ & $(0,1)$ & global & \cmark & \cmark \\
	& Peng \textit{et al.} \cite{peng_arock_2016} & edge-based & $(0,1)$ & \xmark & \cmark (one node at a time) & \xmark \\
	& \bf This paper & \bf edge-based & $\pmb{(0,1)}$ & \textbf{local} & \pmb{\cmark} & \pmb{\cmark} \\
\hline 
\multicolumn{7}{l}{$\dagger$ The results presented in these works do not explicitly address master-slave architectures, see Remark~\ref{rem:Conv-rate}.}
\end{tabular}
\end{table*}

In this section, under the assumption that the local functions $f_i$ are strongly convex, we prove that the mean convergence of Algorithm~\ref{alg:robust-distributed-admm} is locally linear. Moreover we provide an upper bound to the convergence rate.

In the scenario of Assumption~\ref{as:asynchronous-lossy-scenario}, it is not guaranteed that the auxiliary variables $z_{ij}$ are updated at each iteration $k$. Indeed, by introducing the random diagonal matrix $\B(k)\,\in \, \R^{nM \times nM}$ such that
$$
	\B(k) = \diag\{ \beta_{ij}(k) I_n \ \forall i \in \mathcal{V}, \ \forall j \in \mathcal{N}_i \},
$$
we can rewrite~\eqref{eq:affine-perturbed-z} as
\begin{equation}\label{eq:randomized-perturbed-update}
\begin{split}
	&\z(k+1) = \hat{\T}(k) \z(k) + \\
	&\qquad\qquad\,\,\, + \B(k) \left[ \uv + \o'(\x(k+1) - \x^*) \right]
\end{split}
\end{equation}
where  $\hat{\T}(k) := \I - \B(k) (\I - \T)$, with $\T$ defined in \eqref{eq:T}. This allows us to interpret Algorithm~\ref{alg:robust-distributed-admm} as the application of a randomized and perturbed affine operator.
 
The goal of this section is to evaluate the behavior of the mean error $\mathbb{E} [\norm{\x(k) - \x^*}]$ as $k \to \infty$ and, in particular, showing that it converges to zero linearly.
%By iterating \eqref{eq:randomized-perturbed-update}, exploiting property \eqref{eq:Bound_primal}, taking the expected value and using Jensen's inequality, we get the following inequality
The following inequality holds, see Appendix~\ref{app:main-inequality} for the proof:
\begin{align}\label{eq:MainInequality}
	&\expval{\norm{\x(k+1) - \x^*}} \leq \nonumber\\
	&\,\, \zeta\, \sqrt{\expval{\norm{\A^\top  \, \prod_{\ell = 0}^{k-1} \hat{\T}(\ell)}^2}} \norm{\z(0) - \bar{\z}} +  \nonumber \\ 
	&\,\, + \zeta\,\sum_{h = 0}^{k-1}\,  \sqrt{\expval{\norm{ \A^\top  \prod_{\ell = h+1}^{k-1} \hat{\T}(\ell)}^2}} \expval{\norm{\B(h)}}\, \times \\ &\,\,\, \times \expval{\norm{\o'(\x(h+1) - \x^*)}}. \nonumber
\end{align}

Next we show that the square of the first term on the right-hand side of~\eqref{eq:MainInequality} converges to zero linearly as $k \to \infty$.
Notice that this term can be rewritten as
$$
	\Dm(k) = \mathbb{E}\left[  \left(\prod_{\ell = 0}^{k-1} \hat{\T}(\ell)\right)^\top \A \A^\top \left(\prod_{\ell = 0}^{k-1} \hat{\T}(\ell)\right)\right],
$$
if $k\geq 1$, otherwise $\Dm(0) =  \A \A^\top $. A simple recursive argument shows that
$$
	\Dm(k+1) = \mathbb{E}\left[\hat{\T}^\top(0) \Dm(k) \hat{\T}(0) \right]
$$
that is, $\Dm(k)$ is the evolution of a linear dynamical system which can be written in the form
\begin{equation}\label{eq:IterationDelta}
\Dm(k+1)=\mathcal{L}(\Dm(k))
\end{equation}
where $\mathcal{L}\,:\,\R^{nM \times nM} \to \R^{nM\times nM}$ is defined by
$$
 \mathcal{L}(\mathbf{M})= \mathbb{E}\left[\hat{\T}^\top(0) \,\mathbf{M}\, \hat{\T}^\top(0) \right].
$$
The spectral properties of $\mathcal{L}$ have been characterized in Lemma~\ref{lem:T-randomized-properties} in Appendix \ref{app:randomized-linear_convergence}. In particular the following two facts have been established. First, if $\T$ has $H$ semi-simple eigenvalues in $1$ then $ \mathcal{L}$ has $H^2$ semi-simple  eigenvalues in $1$, while all the other eigenvalues are strictly inside the unitary circle. Second, the matrix $\A \A^\top$ belongs to the eigenspace generated be the eigenvectors corresponding to the eigenvalues strictly smaller than $1$. These two facts directly imply the following result.
\begin{proposition}\label{prop:L-to-0}
Consider \eqref{eq:IterationDelta} with $\Dm(0)=\A \A^\top$. Then, there exists $C' > 0$ such that
$$
	\norm{\Dm(k)} \leq C' \bar{\gamma}_{\mathrm{M}}^k,
$$
where
\begin{equation}\label{eq:bar-gamma}
	\bar{\gamma}_{\mathrm{M}} = \max \left\{|\gamma| \,: \, \gamma \in \Lambda(\mathcal{L}),\, \gamma \neq 1\right\}.
\end{equation}
\end{proposition}

The previous Proposition states that the convergence rate to zero of $\pmb{\Delta}(k)$ is upper bounded by the largest eigenvalue in absolute value of $\mathcal{L}$ different from $1$. One can show that $\bar{\gamma}_{\mathrm{M}}$ is a suitable upper-bound also for the convergence rate to zero of the second term in the right-hand side of \eqref{eq:MainInequality}.
Indeed, the following Proposition holds true in a neighborhood of the optimal solution.

\begin{proposition}\label{pr:randomized-linear-convergence}
Assume that the local costs $f_i$ are strongly convex and twice continuously differentiable, and that Assumption~\ref{as:asynchronous-lossy-scenario} holds. Then there exists $\epsilon > 0$ such that, if $\operatorname{dist}(\z(0), \fix(\Tpr)) \leq \epsilon$, then Algorithm \ref{alg:robust-distributed-admm} converges linearly -- in mean -- to the optimal solution, \textit{i.e.}, 
$$
	\mathbb{E} \left[ \|x_i(k)-x^*\| \right] \leq C \gamma^k \qquad i \in \mathcal{V}, 
$$
with $C > 0$, and $0 \leq \gamma \leq \sqrt{\bar{\gamma}_{\mathrm{M}}} < 1$.
\end{proposition}
\begin{proof}
See~Appendix~\ref{app:randomized-linear_convergence}.
\end{proof}

All the details of the derivation can be found in Appendix~\ref{app:randomized-linear_convergence}. The next Proposition provides a matricial characterization of the operator $\mathcal{L}$ that can be used to compute $\bar{\gamma}_{\mathrm{M}}$. 

\begin{proposition}\label{pr:randomized-convergence-rate}
The linear operator $\mathcal{L}$ can be equivalently described by the following matrix
$$
	\L = \expval{\hat{\T}(0) \otimes \hat{\T}(0)}
$$
which is equal to
\begin{align*}
	\L &= \I \otimes \I - \I \otimes \mathbb{E}[\B(0)] + \I \otimes \mathbb{E}[\B(0)]\T - \mathbb{E}[\B(0)] \otimes \I + \\
	&+ \mathbb{E}[\B(0)]\T \otimes \I + \mathbb{E}[\B(0) \otimes \B(0)] (\I - \T) \otimes (\I - \T).
\end{align*}
Hence 
$$
	\bar{\gamma}_{\mathrm{M}} = \max \left\{|\gamma| \,: \, \gamma \in \Lambda(\L),\, \gamma \neq 1\right\}.
$$
\end{proposition}
\begin{proof}
See~Appendix~\ref{app:bound-convergence-rate}.
\end{proof}

Observe that, from Assumption \ref{as:asynchronous-lossy-scenario}, it follows that $\mathbb{E}[\B(0)]$ and $\mathbb{E}[\B(0) \otimes \B(0)]$ are both diagonal matrices. In particular, when considering the uniform scenario introduced in Remark~\ref{rem:uniform-rv}, the computation of $\L$ simplifies to 
\begin{align*}
	\L &= (1 - 2p_\beta) \I \otimes \I + p_\beta \big[ \I \otimes \T + \T \otimes \I \big] + \\ & + \mathbb{E}[\B(0) \otimes \B(0)] (\I - \T) \otimes (\I - \T).
\end{align*}
since $\mathbb{E}[\B(0)] = p_\beta \I$, and where the diagonal elements of $\mathbb{E}[\B(0) \otimes \B(0)]$ are given by
$$
	\mathbb{E}[\beta_{ij}(0)\beta_{hl}(0)] = \begin{cases}
		p_\beta^2 & \quad \text{if}\ i \neq h, \\
		p_\beta^2/p_\mu & \quad \text{if}\ i=h,\ j \neq l, \\
		p_\beta & \quad \text{if}\ i = h,\ j = l. \\
	\end{cases}
$$

We conclude this Section with the following Remarks that emphasize some interesting properties of Algorithm \ref{alg:robust-distributed-admm}.

\begin{remark}[\bf Quadratic case: global linear convergence]
If the local costs are quadratic, then the linear convergence results of Proposition \ref{pr:local-linear-convergence}~and~\ref{pr:randomized-linear-convergence} hold globally. This is a consequence of the fact that the auxiliary variable update~\eqref{eq:randomized-perturbed-update} characterizing the proposed algorithm becomes a (randomized) affine update:
$$
	\z(k+1) = \hat{\T}(k) \z(k) + \B(k) \uv.
$$
\end{remark}

\begin{remark}[\bf Convergence of randomized (R-)ADMM]
Owing to the operator theoretical interpretation of the R-ADMM, its convergence in the presence of asynchronous updates and packet losses can be guaranteed almost surely for all choices of initial conditions and of the free parameters $\alpha$ and $\rho$. On the other hand, proving convergence of the augmented Lagrangian-based interpretation of R-ADMM (see Remark~\ref{rem:Lagrangian-based}) is not as straightforward. To the best of our knowledge, \cite{majzoobi_analysis_2018} is the only paper proving convergence of the standard ADMM in the presence of packet losses. Interestingly, building on the framework established by \cite{shi_linear_2014}, the authors of \cite{majzoobi_analysis_2018} have proved global convergence of the ADMM in the presence of uniformly distributed packet losses. However no linear convergence has been established and, in turn, no characterization of the rate of convergence has been provided. Moreover, asynchronous scenarios have not been analyzed. On the other hand, in \cite{chang_asynchronous_2016_2} global linear convergence of ADMM is shown in the presence of asynchronous updates, but the results hold only for master-slave architectures.

Concerning the general R-ADMM algorithm, results for lossy and asynchronous scenarios have been obtained in \cite{peng_arock_2016, hannah_unbounded_2018,combettes_stochastic_2015,bianchi_coordinate_2016,combettes_stochastic_2019}. More precisely, in \cite{peng_arock_2016, hannah_unbounded_2018}, it is proved that the ARock algorithm converges sub-linearly with asynchronous updates and (possibly unbounded) transmission delays. However convergence is guaranteed only if a single agent updates at each iteration, while Algorithm~\ref{alg:robust-distributed-admm} is fully parallel, \textit{i.e.} guarantees convergence when an arbitrary number of agents updates simultaneously. Moreover, as already stressed in Remark \ref{rem:ARock}, the presence of a common memory among the nodes makes the ARock framework not suitable to theoretically analyze the convergence properties of Algorithm \ref{alg:robust-distributed-admm}.

\noindent In \cite{bianchi_coordinate_2016,combettes_stochastic_2015} the convergence of the general R-ADMM has been shown in the presence of randomized coordinate updates. Furthermore, \cite{combettes_stochastic_2019} proves the global, linear convergence of the randomized R-ADMM under the assumption that $\A$ is full row rank. We remark that the the linear convergence result of \cite{combettes_stochastic_2019} applies to the distributed setup of interest only for master-slave topologies, for which $\A$ is full row rank [cf. Remark~\ref{rem:Conv-rate}].

The review literature provided in this Remark and in Remark \ref{rem:Conv-rate} has been conveniently summarized in Table \ref{tab:convergence-results-comparison}.
\end{remark}

\begin{remark}[\bf Stable parameters pairs]
Observe that both Proposition~\ref{cr:convergence}, for the case of reliable communications, and Proposition~\ref{cr:convergence_lossy}, for the randomized updating scenario, establish convergence provided that $0<\alpha<1$ and $\rho>0$. However, these conditions are \textit{only sufficient and not necessary} and, in particular, the convergence might hold also for values of $\alpha \geq 1$. This fact, proved in \cite{giselsson_linear_2017} under the assumption that $\A$ is full row rank, can be empirically observed in Section~\ref{sec:simulation} where, for the case of quadratic functions $f_i,\ i\in\mathcal{V}$, the region of attraction in parameter space is larger. Moreover, despite what the intuition would suggest, the larger the packet loss uniform probability $p_\lambda$, the larger the region of convergence. However, this increased region of stability is counterbalanced by a slower convergence rate of the algorithm.
\end{remark}

%--------------------------------------------------------------------------------------------------------------------------
\section{Simulations}\label{sec:simulation}
In this Section we present numerical results that showcase the convergence properties of the proposed Algorithm~\ref{alg:robust-distributed-admm} in different scenarios.

%-------------------------------------------------------------------
\subsection{Error trajectories}
We consider a random geometric graph with $N = 25$ nodes, and quadratic costs $f_i(x_i) = (1/2) x_i^\top Q_i x_i - \langle r_i, x_i \rangle$, with $Q_i = Q_i^\top \succ 0$, and $n = 5$. We performed a set of Monte Carlo simulations, each $500$ iterations long and averaging over $100$ realizations of the uniformly distributed packet loss and update random variables.

Fig.~\ref{fig:varying-alpha} depicts the logarithmic error $\log \norm{\x(k) - \x^*}$ for different values of $\alpha$ when both packet losses and asynchronous updates are present. First of all we notice that, the convergence is linear. Moreover, the closer $\alpha$ is to $1$, the faster the convergence is; notice that, although Proposition~\ref{cr:convergence_lossy} does not guarantee convergence for $\alpha = 1$, this is nonetheless achieved. This result suggests that the R-ADMM is advantageous w.r.t. the standard ADMM, thus justifying its choice.
\begin{figure}[!ht]
	\centering
	\includegraphics[width=0.49\textwidth]{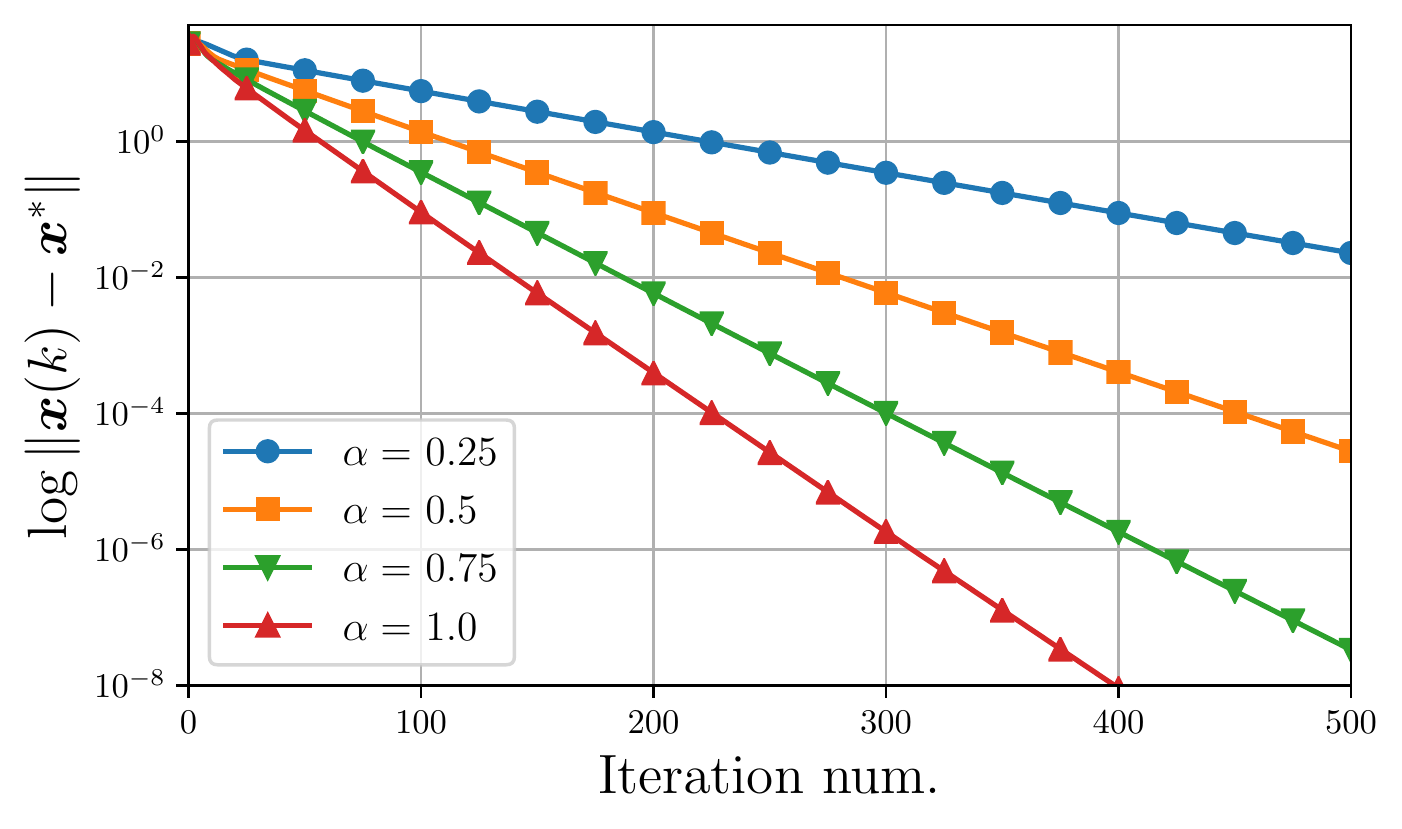}
	\caption{Logarithmic error for different values of $\alpha$; with $p_\lambda = 0.4$, $p_\mu = 0.8$, $\rho = 1$.}
	\label{fig:varying-alpha}
\end{figure}

Fig.~\ref{fig:varying-losses} depicts the logarithmic error for different values of the packet loss probability $p_\lambda$. The result is that the larger $p_\lambda$ is, the slower the convergence, since the number of updates performed at each iteration decreases.
\begin{figure}[!ht]
	\centering
	\includegraphics[width=0.49\textwidth]{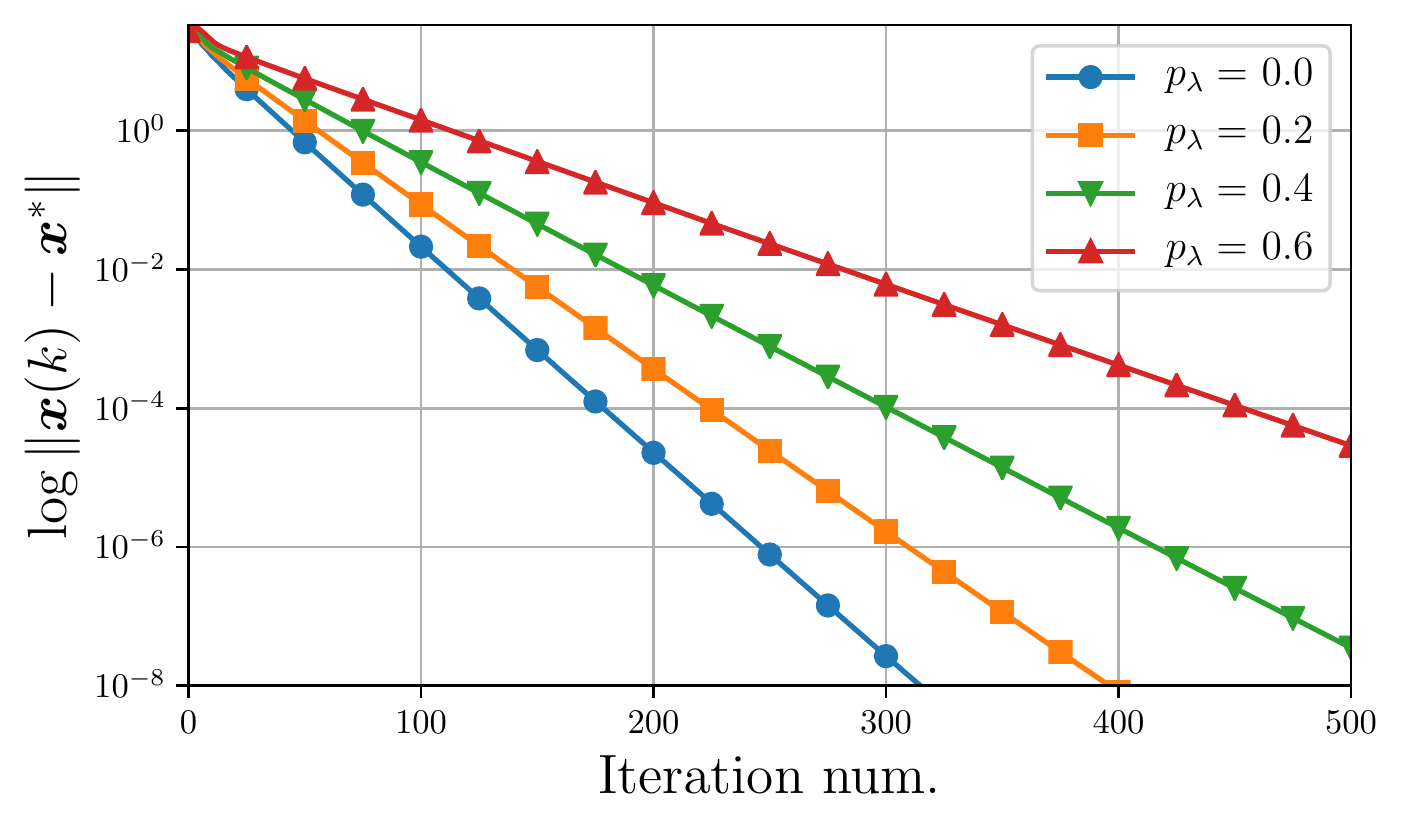}
	\caption{Logarithmic error for different values of $p_\lambda$; with $p_\mu = 0.8$, $\alpha = 0.75$, $\rho = 1$.}
	\label{fig:varying-losses}
\end{figure}

Finally, Fig.~\ref{fig:stability-boundaries} depicts the stable pairs of values $(\rho, \alpha)$ for different packet loss probabilities and $p_\mu = 1$. A pair $(\rho, \alpha)$ is considered stable if it leads to convergence of Algorithm~\ref{alg:robust-distributed-admm} over all of the Monte Carlo iterations. The curves in Fig.~\ref{fig:stability-boundaries} represent the upper bound to the value of $\alpha$ that gives stable pairs.
\begin{figure}[!ht]
	\centering
	\includegraphics[width=0.49\textwidth]{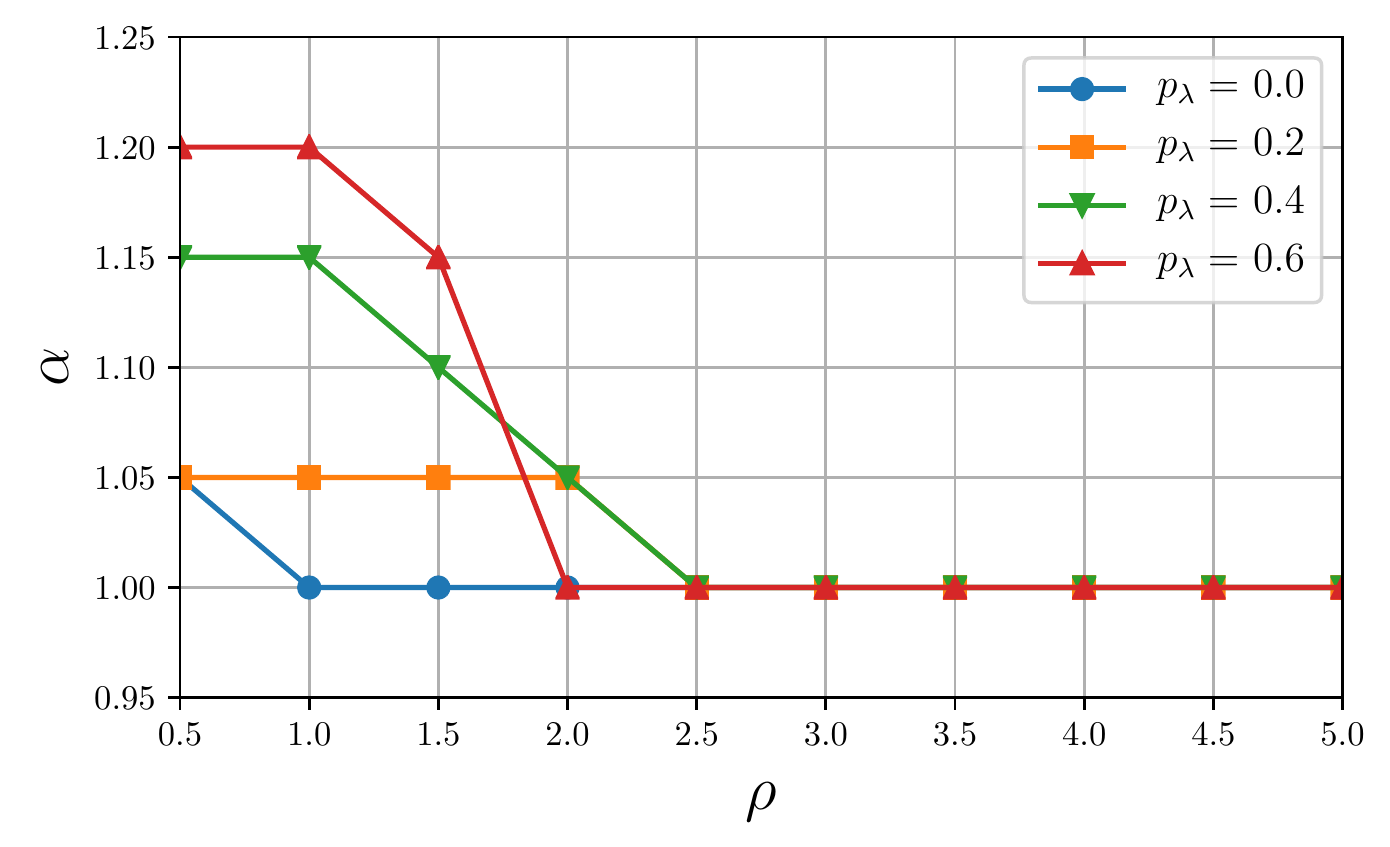}
	\caption{Stability boundaries in the $(\rho, \alpha)$ plane for different values of $p_\lambda$. Below each curve are included the stable pairs of parameters, which lead to convergence over all of the Monte Carlo simulations.}
	\label{fig:stability-boundaries}
\end{figure}
An interesting feature of the proposed algorithm is that the larger the packet loss probability is, the larger the stability region. This is however counterbalanced by the fact that the convergence rate increases as the packet loss grows larger [cf. Fig.~\ref{fig:varying-losses}].

%-------------------------------------------------------------------
\subsection{Convergence rate}
We consider now a random geometric graph with $N = 5$ nodes and $n = 2$, the same quadratic cost for each agent, and for simplicity $p_\mu = 1$. We evaluate the empirical convergence rate of the R-ADMM $\hat{\gamma}$, computed as the slope of the logarithmic error trajectory averaged over $100$ Monte Carlo simulations, each $1000$ iterations long. Fig.~\ref{fig:convergence-rate} depicts the results.
\begin{figure}[!ht]
	\centering
	\includegraphics[width=0.49\textwidth]{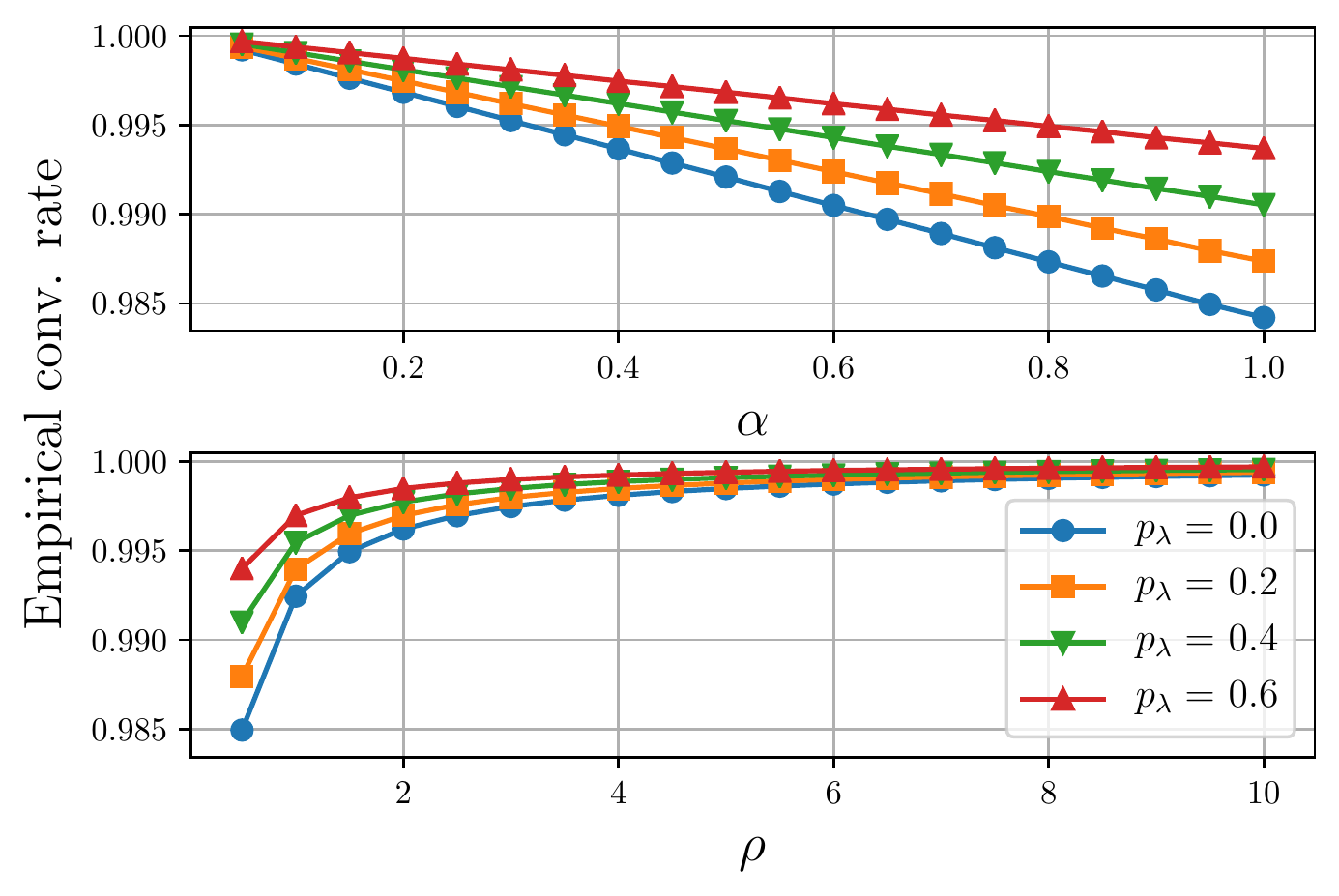}
	\caption{Empirical convergence rate for different values of the parameters $\alpha$ (top) and $\rho$ (bottom), with $N = 5$ and $n = 2$. For the top figure we choose $\rho = 0.5$, and for the bottom we choose $\alpha = 0.95$.}
	\label{fig:convergence-rate}
\end{figure}
Notice that, as evidenced also by Fig.~\ref{fig:varying-alpha} above, the larger $\alpha$ is, the lower the convergence rate. On the other hand, in this particular scenario the larger is $\rho$, the worse the convergence rate.

Moreover, for each choice of $\alpha \in (0,1)$, $\rho \in [0.5, 10]$ and $p_\lambda \in \{0, 0.2, 0.4, 0.6\}$, we computed $\bar{\gamma}_{\mathrm{M}}$, which by Proposition~\ref{pr:randomized-linear-convergence} gives a bound to the convergence rate that holds in mean. Indeed, as evidenced by Tab.~\ref{tab:convergence-rates}, the bound appears to be extremely tight, with a maximum difference that is less than $1$\textperthousand.
\begin{table}[!ht]
	\centering
	\caption{Difference between empirical and theoretical convergence rate.}
	\label{tab:convergence-rates}
	\begin{tabular}{cccc}
	\hline 
	 & Maximum & Minimum & Mean $\pm$ Std \\ 
	\hline 
	$\frac{|\hat{\gamma} - \bar{\gamma}_{\mathrm{M}}|}{\bar{\gamma}_{\mathrm{M}}}$ & $4.9 \times 10^{-5}$ & $8.28 \times 10^{-11}$ & $(1.1 \pm 2.8) \times 10^{-6}$ \\ 
	\hline 
	\end{tabular} 
\end{table}

Finally, Fig.~\ref{fig:convergence-rate-nodes} depicts the empirical rate for complete graphs with different numbers of nodes, and for different packet loss probabilities. The cost is quadratic and equal for all the agents.
\begin{figure}[!ht]
	\centering
	\includegraphics[width=0.49\textwidth]{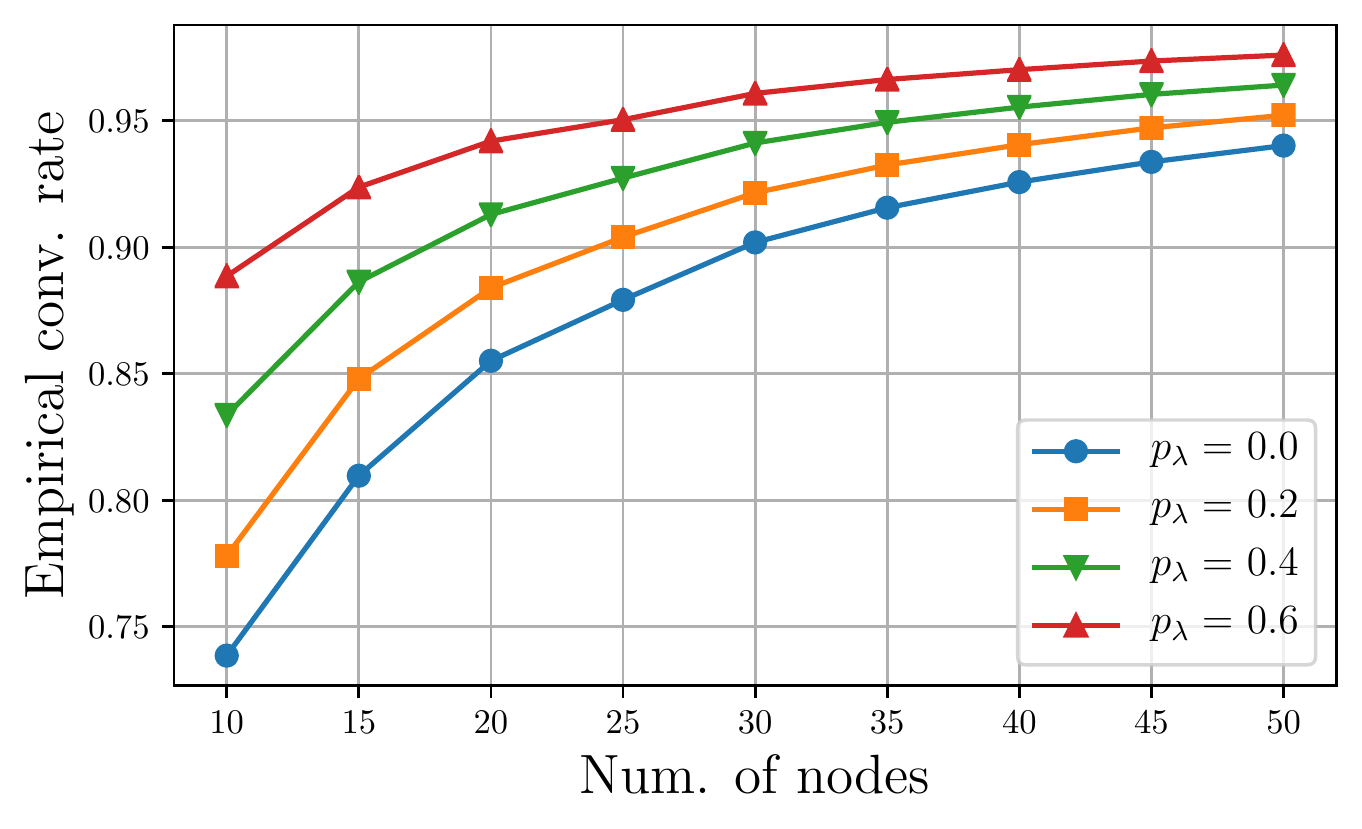}
	\caption{Empirical convergence rate for complete graphs of different size and for different $p_\lambda$, with $p_\mu = 0.8$, $n = 2$, $\alpha = 0.95$, $\rho = 0.5$.}
	\label{fig:convergence-rate-nodes}
\end{figure}
As remarked above, the larger $p_\lambda$ is, the larger $\hat{\gamma}$; and, in this particular case, the convergence rate degrades monotonically with the number of nodes in the graph.

%-------------------------------------------------------------------
\subsection{Quartic function}
In order to study the effect of curvature on the convergence of the algorithm, we considered a random graph with $N = 10$, $n = 1$, and all costs equal to
$$
	f_i(x) = \frac{x^4}{12} + q \frac{x^2}{2}, \qquad x \in \R, \ \ q > 0.
$$
Fig.~\ref{fig:quartic} depicts the logarithmic error for different values of the parameter $q$. Clearly, the curvature of the cost may deeply affect the rate of convergence.
\begin{figure}[!ht]
	\centering
	\includegraphics[width=0.49\textwidth]{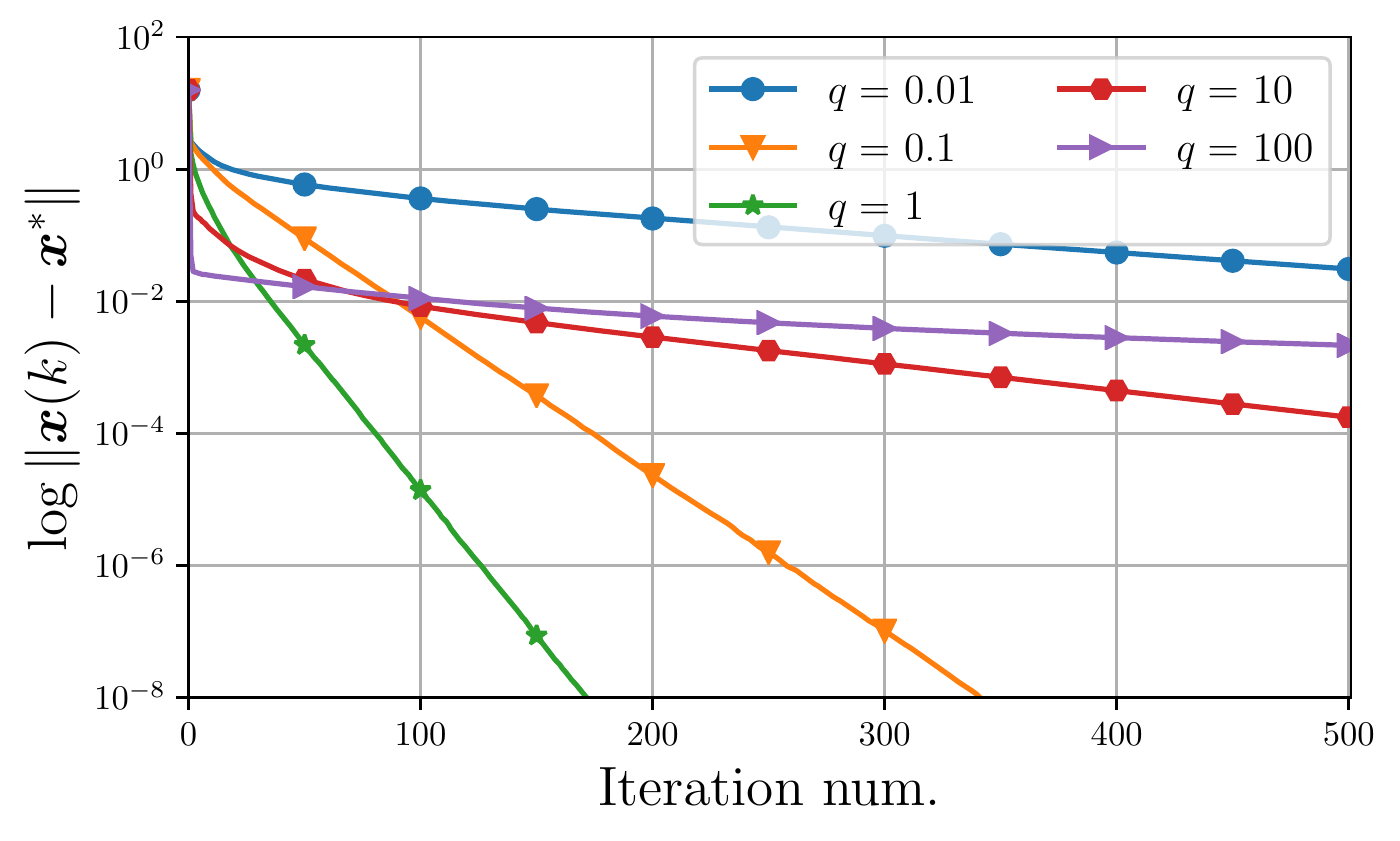}
	\caption{Logarithmic error for different values of $q$, with $p_\lambda = 0.5$ $p_\mu = 0.8$, $\alpha = 0.95$, $\rho = 0.5$.}
	\label{fig:quartic}
\end{figure}
Notice moreover that the convergence is globally linear, although the theoretical results guarantee only local linear convergence.

%--------------------------------------------------------------------------------------------------------------------------
\section{Conclusions and Future Directions}\label{sec:conclusions}
In this paper we addressed distributed convex optimization problems over peer-to-peer networks with both unreliable communications and asynchronous updates of the nodes. We proposed a modified version of the relaxed ADMM that, exploiting operator theoretical results, can be shown to converge almost surely. Moreover, by further assuming the local costs to be strongly convex, we proved local linear mean convergence of the proposed algorithm.

\noindent We have cast the proposed algorithm in the context of the literature and discussed its novelty. And finally, we have presented interesting numerical results that showcase the resilience and robustness of the proposed algorithm.

%--------------------------------------------------------------------------------------------------------------------------
\appendices

\ifarxiv

%--------------------------------------------------------------------------------------------------------------------------
\section{Derivation of~\eqref{eq:r-admm}}\label{app:derivation-splitting-admm}
The Peaceman-Rachford splitting~\eqref{eq:iterates-splitting} applied to the dual problem~\eqref{eq:admm-dual-problem} is characterized by the updates
\begin{subequations}\label{eq:prs-three-eqs}
\begin{align}
	w(k+1) &= \prox_{\rho d_f}(z(k)) \label{eq:prs-three-eqs-w} \\
	v(k+1) &= \prox_{\rho d_g}(2w(k+1) - z(k)) \label{eq:prs-three-eqs-v} \\
	z(k+1) &= z(k) + 2\alpha (v(k+1) - w(k+1)).
\end{align}
\end{subequations}
We show now that~\eqref{eq:prs-three-eqs-w} is equivalent to~\eqref{eq:r-admm-x}~and~\eqref{eq:r-admm-w}; the same argument can be applied to~\eqref{eq:prs-three-eqs-v}.

By the definition of proximal operator and of $d_f$ it holds
\begin{align}
	&\prox_{\rho d_f}(z(k)) = \argmin_w \left\{ f^*(A^\top w) + \frac{1}{2\rho} \norm{w - z(k)}^2 \right\} \nonumber \\
	&\ = \argmin_w \left\{ \max_x \left\{ \langle A^\top w, x \rangle - f(x) \right\} + \frac{1}{2\rho} \norm{w - z(k)}^2 \right\} \label{eq:prox-df}
\end{align}
where we applied the definition of convex conjugate. Consider now the minimum in~\eqref{eq:prox-df}, we have
\begin{align}
	&\min_w \left\{ \max_x \left\{ \langle A^\top w, x \rangle - f(x) \right\} + \frac{1}{2\rho} \norm{w - z(k)}^2 \right\} = \nonumber \\
	&= \min_w \max_x \left\{ \langle A^\top w, x \rangle - f(x) + \frac{1}{2\rho} \norm{w - z(k)}^2 \right\} \nonumber \\
	&= \max_x \min_w \left\{ \langle A^\top w, x \rangle - f(x) + \frac{1}{2\rho} \norm{w - z(k)}^2 \right\} \label{eq:max-min}
\end{align}
and by the first optimality condition for the innermost minimization it must hold:
\begin{equation}\label{eq:derived-w}
	w = z(k) - \rho A x
\end{equation}
which is exactly~\eqref{eq:r-admm-w}. Substituting~\eqref{eq:derived-w} into~\eqref{eq:max-min} yields
$$
	\eqref{eq:max-min} = \max_x \left\{ \langle z(k) - \rho A x, A x \rangle - f(x) + \frac{\rho}{2} \norm{A x}^2 \right\}
$$
and changing the sign of the cost function we obtain~\eqref{eq:r-admm-x}. \cvd

%--------------------------------------------------------------------------------------------------------------------------
\section{Proof of equivalence between~\eqref{eq:r-admm}~and~\eqref{eq:lagrangian-r-admm}}\label{app:equivalence}
We show that we can derive~\eqref{eq:lagrangian-r-admm} from~\eqref{eq:r-admm}, thus proving their equivalence.

We derive first some equalities that will be useful in the following. By~\eqref{eq:r-admm-w} we have
\begin{equation}\label{eq:z-with-w}
	z(k) = w(k+1) + \rho A x(k+1)
\end{equation}
and using this fact into~\eqref{eq:r-admm-v} yields
\begin{equation}\label{eq:v-with-w}
	v(k+1) = w(k+1) - \rho (A x(k+1) + B y(k+1) - c).
\end{equation}
Moreover, substituting~\eqref{eq:v-with-w} into~\eqref{eq:r-admm-v} we obtain
\begin{equation}
	2w(k+1) - z(k) = w(k+1) - \rho A x(k+1).
\end{equation}

Consider now~\eqref{eq:r-admm-y}, the following chain of equalities holds
\begin{align*}
	&y(k+1) = \argmin_y \Big\{ g(y) + \rho \langle A x(k+1), By \rangle + \\ &\quad- \langle w(k+1), A x(k+1) + B y - c \rangle + \frac{\rho}{2} \norm{By - c}^2 \Big\} \\
	&= \argmin_y \Big\{ f(x(k+1)) + g(y) + \\ &\quad- \langle w(k+1), A x(k+1) + B y - c \rangle + \\ &\quad+ \frac{\rho}{2} \norm{A x(k+1)}^2 + \rho \langle A x(k+1), By \rangle + \frac{\rho}{2} \norm{By - c}^2\Big\} \\
	&= \argmin_y \mathcal{L}_\rho(x(k+1),y,w(k+1))
\end{align*}
where we derived the first by subtracting $\langle w(k+1), A x(k+1) - c \rangle$, the second by adding $f(x(k+1))$, $(\rho / 2) \norm{A x(k+1)}^2$, and $\rho \langle A x(k+1), -c \rangle$, which is allowed since they do not depend on $y$. The third equality holds by definition of square norm and augmented Lagrangian~\eqref{eq:augmented-lagrangian}. We have thus derived~\eqref{eq:lagrangian-r-admm-y}.

Substituting~\eqref{eq:v-with-w} and~\eqref{eq:z-with-w} into~\eqref{eq:r-admm-z} yields
\begin{align}
	z(k+1) &= w(k+1) + \rho A x(k+1) + \nonumber \\ &- 2\alpha\rho (A x(k+1) + B y(k+1) - c) \nonumber \\
	&= w(k+1) - \rho (By(k+1) - c) + \label{eq:z-with-w-bis} \\ &- \rho(2\alpha - 1) (A x(k+1) + B y(k+1) - c) \nonumber
\end{align}
where the second equation was obtained adding and subtracting $\rho (B y(k+1) - c)$. Finally, evaluating~\eqref{eq:z-with-w-bis} at iteration $k$ and using~\eqref{eq:z-with-w} we get~\eqref{eq:lagrangian-r-admm-w}.

From~\eqref{eq:r-admm-x} and using~\eqref{eq:z-with-w-bis} at iteration $k$ we can write
\begin{align*}
	&x(k+1) = \argmin_x \Big\{ f(x) - \langle w(k), A x \rangle + \\ &\quad + \rho(2\alpha - 1) \langle A x(k) + B y(k) - c, A x \rangle + \\ &\quad + \rho \langle B y(k) - c, A x \rangle + \frac{\rho}{2} \norm{A x}^2 \Big\} \\
	&= \argmin_x \Big\{ f(x) + g(y(k)) - \langle w(k), A x + B y(k) - c \rangle + \\ &\quad + \rho(2\alpha - 1) \langle A x(k) + B y(k) - c, A x \rangle + \\ &\quad + \frac{\rho}{2} \norm{B y(k) - c}^2 + \rho \langle B y(k) - c, A x \rangle + \frac{\rho}{2} \norm{A x}^2 \Big\} \\
	&= \argmin_x \Big\{ \mathcal{L}_\rho(x,y(k),w(k)) + \\ &\quad + \rho(2\alpha - 1) \langle A x(k) + B y(k) - c, A x \rangle \Big\}
\end{align*}
where the second equality was derived adding $g(y(k))$, $- \langle w(k), B y(k) - c \rangle$, and $(\rho / 2) \norm{B y(k) - c}^2$ which do not depend on $x$, and the third by the definitions of square norm and augmented Lagrangian. This proves equivalence of~\eqref{eq:r-admm-x} with~\eqref{eq:lagrangian-r-admm-x}.

Finally, by the results we derived above, we can see that if the initial conditions satisfy
$$
	z(0) = w(0) - \rho (2\alpha - 1) (A x(0) + B y(0) - c) - \rho (B y(0) - c).
$$
then the trajectories for $x$ and $y$ generated by the splitting R-ADMM and the Lagrangian R-ADMM coincide. \cvd

\fi

%--------------------------------------------------------------------------------------------------------------------------
\section{Proofs of Section~\ref{sec:synchronous-ADMM}}

%-------------------------------------------------------------------
\subsection{Derivation of~\eqref{eq:distributed-admm}}\label{app:derivation-distributed-admm}
Following the derivation in \cite{peng_arock_2016}, we show that applying the R-ADMM of~\eqref{eq:r-admm} to the distributed problem of interest yields~\eqref{eq:distributed-admm}.

%---------------------------------------
\paragraph*{$x$ update}
Using the particular structure of $\A$ we can see that in update~\eqref{eq:r-admm-x}:
$$
	\norm{\A \x}^2 = \sum_{i=1}^N d_i \norm{x_i}^2
$$
since each $x_i$ appears in $d_i$ constraints, \textit{i.e.} rows of $\A$. Moreover
$$
	\langle \A^\top \z(k), \x \rangle = \sum_{i=1}^N \langle \sum_{j \in \mathcal{N}_i} z_{ij}(k), x_i \rangle
$$
since the $i$-th row of $\A^\top$ sums over the auxiliary variables stored by $i$. Therefore~\eqref{eq:r-admm-x} becomes
$$
	\x(k+1) = \argmin_\x \sum_{i=1}^N \bigg\{ f_i(x_i) - \langle \sum_{j \in \mathcal{N}_i} z_{ij}(k), x_i \rangle + \frac{\rho d_i}{2} \norm{x_i^2} \bigg\}
$$
which is clearly separable over the single components. Moreover, node $i$ has all the information necessary to compute $x_i(k+1)$.

%---------------------------------------
\paragraph*{$y$ update}
Update~\eqref{eq:r-admm-y} in the distributed scenario becomes
$$
	\y(k+1) = \argmin_{\y = \P \y} \left\{ \langle 2\w(k+1) - \z(k), \y \rangle + \frac{\rho}{2} \norm{\y}^2 \right\}
$$
whose KKT conditions are
\begin{subequations}
\begin{align}
	\rho \y &= (\I - \P) \mathbold{\nu} - (2\w(k+1) - \z(k)) \label{eq:kkt-1} \\
	\y &= \P \y \label{eq:kkt-2}
\end{align}
\end{subequations}
where $\mathbold{\nu}$ is the vector of Lagrange multipliers. Plugging~\eqref{eq:kkt-1} into~\eqref{eq:kkt-2} yields
\begin{align}
	\rho \y &= \P (\I - \P) \mathbold{\nu} - \P (2\w(k+1) - \z(k)) \nonumber \\
	&= -(\I - \P) \mathbold{\nu} - \P (2\w(k+1) - \z(k)) \label{eq:kkt-sum}
\end{align}
where the second equality was derived using $\P^2 = \I$, which implies $\P(\I - \P) = - (\I - \P)$.
Finally, summing~\eqref{eq:kkt-1} and~\eqref{eq:kkt-sum} we get
\begin{equation}\label{eq:explicit-y-update}
	\y(k+1) = - \frac{1}{2\rho} (\I + \P) (2\w(k+1) - \z(k))
\end{equation}
which means that $y_{ij}(k+1) = y_{ji}(k+1)$ for any $(i,j) \in \mathcal{E}$ and $k \in \N$.

%---------------------------------------
\paragraph*{$z$ update}
Using~\eqref{eq:r-admm-w}, and substituting~\eqref{eq:explicit-y-update} into~\eqref{eq:r-admm-v}, the auxiliary update~\eqref{eq:r-admm-z} becomes
\begin{equation}\label{eq:z-update-matricial}
	\z(k+1) = (1-\alpha) \z(k) - \alpha \P \z(k) + 2\alpha\rho \P \A \x(k+1)
\end{equation}
and using the definitions of $\A$, $\P$ proves~\eqref{eq:distributed-admm-z}. \cvd

%-------------------------------------------------------------------
\subsection{Proof of Proposition~\ref{pr:local-linear-convergence}}\label{app:linear_convergence}
The proof is divided in the following steps: \textit{(i)} write the auxiliary update of Algorithm~\ref{alg:distributed-admm} as a perturbed affine operator; \textit{(ii)} bound the primal error with the error on the auxiliary variable; \textit{(iii)} show that the primal error converges linearly for a quadratic approximation; \textit{(iv)} extend the result to the general case.

%---------------------------------------
\paragraph*{\textit{(i)} Perturbed affine operator}
From the first order optimality condition for~\eqref{eq:distributed-admm-x} it must hold for any $i \in \mathcal{V}$
\begin{equation}\label{eq:1st-opt-cond}
	\nabla f_i(x_i(k+1)) - [\A^\top \z(k)]_i + \rho d_i x_i(k+1) = 0.
\end{equation}
Therefore using the Taylor expansion of the gradient $\nabla f_i$ around $x^*$ we have
\begin{equation}\label{eq:taylor-expansion}
	\nabla f_i(x_i) = \nabla f_i(x^*) + \nabla^2 f_i(x^*) (x_i - x^*) + o(x_i - x^*)
\end{equation}
for $x \in \mathcal{B}_{x^*}$, where $o : \R^n \to \R^n$ is such that $\norm{o(x_i - x^*)} / \norm{x_i - x^*} \to 0$ as $x_i \to x^*$.
Combining~\eqref{eq:1st-opt-cond}~and~\eqref{eq:taylor-expansion}, the latter evaluated in $x_i(k+1)$, yields
\begin{equation*}
\begin{split}
	&[\A^\top \z(k)]_i - \rho d_i x_i(k+1) = \nabla f_i(x^*) + \\ &\qquad + \nabla^2 f_i(x^*) (x_i(k+1) - x^*) + o(x_i(k+1) - x^*)
\end{split}
\end{equation*}
and solving for $x_i(k+1)$ we get
\begin{align}\label{eq:affine-x_i-update}
	x_i(k+1) &= (\rho d_i I_n + \nabla^2 f_i(x^*))^{-1} \Big[ [\A^\top \z(k)]_i + \\ & + \nabla^2 f_i(x^*) x^* - \nabla f_i(x^*) + o(x_i(k+1) - x^*) \Big]. \nonumber
\end{align}
Stacking the updates~\eqref{eq:affine-x_i-update} for $i \in \mathcal{V}$ we can write
\begin{equation}\label{eq:affine-x-update}
	\x(k+1) = \Hm^{-1} \Big[ \A^\top \z(k) + \g + \o(\x(k+1) - \x^*) \Big]
\end{equation}
where $\Hm = \operatorname{blk\,diag} \left\{ \rho d_i I_n + \nabla^2 f_i(x^*) \right\}$, $\g$ and $\o$ stack $\nabla^2 f_i(x^*) x^* - \nabla f_i(x^*)$ and $o(x_i(k+1) - x^*)$, respectively.

Using the auxiliary update~\eqref{eq:z-update-matricial} and~\eqref{eq:affine-x-update} we can write
\begin{align}
	\z(k+1&) = (1-\alpha) \z(k) - \alpha \P \z(k) + \nonumber \\ &+ 2\alpha\rho \P \A \Hm^{-1} \Big[ \A^\top \z(k) + \g + \o(\x(k+1) - \x^*) \Big] \nonumber \\
	&=: \T \z(k) + \uv + \o'(\x(k+1) - \x^*) \label{eq:perturbed-affine-update}
\end{align}
where $\T=(1-\alpha) \I - \alpha \P + 2\alpha\rho \P \A \Hm^{-1} \A^\top $, $\uv=2\alpha\rho \P \A \Hm^{-1} \g$, and $\o' : \R^{nN} \to \R^{nM}$, $\o'(\cdot)= 2\alpha\rho \P \A \Hm^{-1}\, \o(\cdot)$, decays faster than the argument.

\begin{remark}\label{rem:symmetry-T}
Using the particular structure of $\A$ we see that $\A \Hm^{-1} \A^\top = \operatorname{blk\,diag} \left\{ \1_{d_i \times d_i} \otimes (\rho d_i I_n + \nabla^2 f_i(x^*)) \right\}$. But since $f_i \in \mathcal{C}^2$ for any $i$, the Hessians $\nabla^2 f_i(x^*)$ are symmetric and thus $\T$ is symmetric as well.
\end{remark}

%---------------------------------------
\paragraph*{\textit{(ii)} Primal error bound}
We start by stating the following result.
\begin{lemma}\label{lem:x_i-prox}
The update \eqref{eq:distributed-admm-x} can be rewritten as 
$$
	x_i(k+1)=\prox_{f_i / (\rho d_i)} \Big([\A^\top \z(k)]_i / (\rho d_i) \Big).
$$
\end{lemma}
\begin{proof}
Adding the term $\norm{[\A^\top \z(k)]_i}^2 / (\rho d_i)$ -- which does not depend on $x_i$ -- to the objective function in~\eqref{eq:distributed-admm-x} and using the definition of norm we can write
\begin{align*}
	x_i(k+1) &= \argmin_{x_i} \left\{ f_i(x_i) + \frac{\rho d_i}{2} \norm{x_i - \frac{1}{\rho d_i} [\A^\top \z(k)]_i}^2 \right\} \\
	&= \prox_{f_i / (\rho d_i)} \Big([\A^\top \z(k)]_i / (\rho d_i) \Big)
\end{align*}
where the second equality follows from definition of proximal operator, see section~\ref{subsec:operator-theory}.
\end{proof}
\noindent Denote by $m_i$ the strong convexity modulus of function $f_i$, then we know that the proximal $\prox_{f_i / (\rho d_i)}$ is $ 1 / (1+m_i / (\rho d_i))$-contractive \cite{giselsson_diagonal_2014}, which implies
$$
	\norm{x_i(k+1) - x^*} \leq \frac{1}{m_i + \rho d_i} \norm{[\A^\top (\z(k) - \bar{\z})]_i}
$$
for any $\bar{\z} \in \fix(\Tpr)$. Then
\begin{align}
	\| \x(k+1) - \x^* \|^2 	&\leq \sum_{i=1}^N \frac{1}{(m_i + \rho d_i)^2} \norm{[\A^\top (\z(k) - \bar{\z})]_i}^2 \nonumber \\
	&\leq \zeta^2 \norm{\A^\top (\z(k) - \bar{\z})}^2, \label{eq:primal-error-bound}
\end{align}
where
\begin{align}\label{eq:zeta}
\zeta=\max_i \left\{ 1 / (m_i + \rho d_i) \right\}.
\end{align}

%---------------------------------------
\paragraph*{\textit{(iii)} Linear convergence (quadratic case)}
Assume that the functions $f_i$ are quadratic and, more specifically, \eqref{eq:taylor-expansion} holds true with the residual equal
to $0$.
In this case the Peaceman-Rachford operator is affine and averaged, and the auxiliary update becomes $\z(k+1) = \T \z(k) + \uv$.
The following result characterizes the spectral properties of $\T$.
\begin{lemma}\label{lem:T_properties}
The eigenvalues of $\T$ are either equal to $1$ or strictly inside the unitary circle. Moreover the eigenvalues in $1$ are all semi-simple. In addition the following property holds
\begin{equation}\label{eq:KernelsInc}
	\ker(\I - \T) \subset \ker(\A^\top).
\end{equation}
\end{lemma}
\begin{proof}
For affine averaged operators the eigenvalues of $\T$ are all inside the circle on the complex plane with center $1-\alpha + i0$ and radius $\alpha$ \cite{iutzeler_generic_2019}. This implies that the unique eigenvalues of $\T$ with unitary absolute value are in $1$, and by convergence of the \km they are semi-simple.

\noindent Now let $\ker(\I - \T) = \operatorname{span} \{ \v_1, \ldots, \v_H \}$ where $H$ is the algebraic (and geometric) multiplicity of $1$.
Notice that, given $\bar{\z} \in \fix(\Tpr)$ and $v_h \in \ker(\I - \T)$, then $\bar{\z} + c \v_h \in \fix(\Tpr)$ for any $c \in \R$. For any $\bar{\z} \in \fix(\Tpr)$ from~\eqref{eq:affine-x-update} we have $\x^* = \Hm^{-1} [\A^\top \bar{\z} + \g]$, which, by uniqueness of $\x^*$, implies
\begin{align*}
	\x^* = \Hm^{-1} [\A^\top \bar{\z} + \g] = \Hm^{-1} [\A^\top (\bar{\z} + c \v_h) + \g].
\end{align*}
By the nonsingularity of $\Hm$ this condition implies $\v_h \in \ker(\A^\top)$, $h = 1, \ldots, H$, and thus $\ker(\I - \T) \subset \ker(\A^\top)$.
\end{proof}
Now by iterating $\z(k+1) = \T \z(k) + \uv$ we have
$$
	\z(k) = \T^{k} \z(0) + \sum_{\ell=0}^{k-1} \T^{k-1-\ell} \uv
$$
which is satisfied also by any $\bar{\z}  \in \fix(\Tpr)$. Thus
$
	\z(k) - \bar{\z} = \T^{k} (\z(0) - \bar{\z}),
$
and combining this fact with~\eqref{eq:primal-error-bound} yields
\begin{align*}
	\norm{\x(k+1) - \x^*} %&\leq \zeta \norm{\A^\top \T^{k} (\z(0) - \bar{\z})} \\
	&\leq \zeta \norm{\A^\top \T^{k}} \norm{\z(0) - \bar{\z}}.
\end{align*}
Since $\ker(\I - \T) \subset \ker(\A^\top)$, this implies that $\norm{\A^\top \T^{k}} \leq C' \gamma_{\mathrm{M}}^{k}$ where $\gamma_{\mathrm{M}}$ is the largest eigenvalue in absolute value different from $1$ of $\T$ and where
$C' > 0$, see \cite{iutzeler_generic_2019}. This proves the linear convergence of the primal error in the quadratic case.

%---------------------------------------
\paragraph*{\textit{(iv)} Linear convergence (general case)}
Since the PRS operator is nonexpansive, it follows that $\|\z(k+1) - \z^*\| \leq \|\z(k) - \z^*\|$ and, in turn, that also the sequence $\norm{\x(k) - \x^*}$ is bounded.
Since $\lim_{k \to \infty} \norm{\x(k+1) - \x^*} = 0$, by the definition of $\o'$ we can argue that there exists a sequence of positive numbers $\{ \delta_k \}_{k \geq 1}$ such that $\delta_{k+1} \leq \delta_k$, $\lim_{k \to \infty} \delta_k = 0$ and
$$
	\norm{\o'(\x(k+1) - \x^*)} \leq \delta_{k+1} \norm{\x(k+1) - \x^*}.
$$
Therefore, by iterating \eqref{eq:perturbed-affine-update} and exploiting \eqref{eq:primal-error-bound} 
 the primal error, for $k \geq 1$, can be bounded as
\begin{align*}
	& \norm{\x(k+1) - \x^*} \leq \\
	& \qquad C'' \gamma_{\mathrm{M}}^{k} +  \sum_{\ell = 0}^{k-1} \zeta C' \gamma_{\mathrm{M}}^{k-1-\ell} \delta_{\ell+1} \norm{\x(\ell+1) - \x^*}
\end{align*}
where $C'' = \zeta C' \norm{\z(0) - \bar{\z}}$.
Consider now the sequence $\{ e(k) \}_{k \geq 1}$ such that $e(1) = \norm{\x(1) - \x^*}$, $e(2) = C'' \gamma_{\mathrm{M}} + \zeta C' \delta_1 \norm{\x(1) - \x^*}$ and
\begin{equation}\label{eq:e_k+1}
	e(k+1) = (\gamma_{\mathrm{M}} + \zeta C' \delta_k) e(k), \quad k \geq 2.
\end{equation}
Recalling the definition of $\o'$, we know that there exists a ball $\mathcal{B}_{x^*}$ centered in $x^*$ such that, if $x_i(1)$ belongs to $\mathcal{B}_{x^*}$, $i \in \mathcal{V}$, then $\norm{\o(\x(1) - \x^*)} \leq \delta_1 \norm{\x(1) - \x^*}$ with $\delta_1$ such that $\gamma_{\mathrm{M}} + \zeta C'\delta_1 < 1$, \textit{i.e.} $\delta_1 < (1-\gamma_{\mathrm{M}}) / (\zeta C')$. In this case, by a standard inductive argument, one can show that $\norm{\x(k) - \x^*} \leq e(k)$ for all $k \geq1$.
Notice that in view of~\eqref{eq:Bound_primal}, there exists $\epsilon>0$ such that if  $\operatorname{dist}(\z(0), \fix(\Tpr)) \leq \epsilon$ then $\x(1)\in \mathcal{B}_{x^*}$. Now from \eqref{eq:e_k+1}
$$
	e(k) = \prod_{\ell = 2}^{k-1} (\gamma_{\mathrm{M}} + \zeta C\delta_\ell) e(2), \quad k \geq 2,
$$
from which, since $|\gamma_{\mathrm{M}} + \zeta C'\delta_\ell| < 1$ for all $\ell$, we get that $\lim_{k \to \infty} e(k) = 0$. We conclude the proof by observing that, for any $\xi > 0$, it holds
\begin{equation}
	\lim_{k \to \infty} \frac{\prod_{\ell=2}^{k-1} (\gamma_{\mathrm{M}} + \zeta C\delta_\ell)}{(\gamma_{\mathrm{M}} + \xi)^{k-3}} = 0. \tag*{$\square$}
\end{equation}

%--------------------------------------------------------------------------------------------------------------------------
\section{Proofs of Section~\ref{sec:robust-ADMM}}

%-------------------------------------------------------------------
\subsection{Proof of Proposition~\ref{cr:convergence_lossy}}\label{app:convergence-lossy}
As discussed in Section~\ref{sec:robust-ADMM}, loss of transmissions and asynchronous updates are taken into account by Algorithm~\ref{alg:robust-distributed-admm} by updating a $z$ auxiliary variable only if new information is available. But since the R-ADMM is the Peaceman-Rachford splitting applied to the dual of~\eqref{eq:admm-problem}, we can interpret Algorithm~\ref{alg:robust-distributed-admm} as a randomized Peaceman-Rachford in which each coordinate of the $\z$ vector is randomly updated with nonzero probability. Therefore the convergence results of \cite[Theorem~3]{bianchi_coordinate_2016} or (a particular case of) \cite[Theorem~3.2]{combettes_stochastic_2015} can be applied to prove almost sure convergence to the dual solution and, in turn, by strong duality, to the primal solution.
\cvd

%-------------------------------------------------------------------
\subsection{Derivation of~\eqref{eq:MainInequality}}\label{app:main-inequality}
By~\eqref{eq:Bound_primal} we have $\expval{\norm{\x(k+1) - \x^*}} \leq \zeta \expval{\norm{\A^\top (\z(k+1) - \bar{\z})}}$, and our goal is to find a bound for the right-hand side. Iterating~\eqref{eq:randomized-perturbed-update} we get
\begin{align}
	\z(k+1) - \bar{\z} &= \prod_{\ell = 0}^{k-1} \hat{\T}(\ell) (\z(0) - \bar{\z}) + \label{eq:z-error}\\
	&+ \sum_{h = 0}^{k-1} \left( \prod_{\ell = h+1}^{k-1} \hat{\T}(\ell) \right) \B(h) \o'(\x(h+1) - \x^*) \nonumber.
\end{align}
Multiplying by $\A^\top$, taking the norm of~\eqref{eq:z-error} and using triangle inequality and submultiplicativity then yields
\begin{align}
	&\norm{\A^\top \z(k+1) - \bar{\z}} \leq \norm{\A^\top \prod_{\ell = 0}^{k-1} \hat{\T}(\ell)} \norm{\z(0) - \bar{\z}} + \\
	&+ \sum_{h = 0}^{k-1} \norm{\A^\top \left( \prod_{\ell = h+1}^{k-1} \hat{\T}(\ell) \right)} \norm{\B(h)} \norm{\o'(\x(h+1) - \x^*)} \nonumber.
\end{align}

Now, taking the expectation we get:
\begin{align}
	&\expval{\norm{\A^\top \z(k+1) - \bar{\z}}} \leq \expval{\norm{\A^\top \prod_{\ell = 0}^{k-1} \hat{\T}(\ell)}} \norm{\z(0) - \bar{\z}} + \nonumber \\
	&+ \sum_{h = 0}^{k-1} \expval{\norm{\A^\top \left( \prod_{\ell = h+1}^{k-1} \hat{\T}(\ell) \right)}} \expval{\norm{\B(h)}} \times \nonumber \\ &\times \expval{\norm{\o'(\x(h+1) - \x^*)}} \label{eq:z-error-expected}.
\end{align}
Finally, by Jensen's inequality for concave functions (as the square root is), we know that
$$
	\expval{\norm{\cdot}} = \expval{\sqrt{\norm{\cdot}^2}} \leq \sqrt{\expval{\norm{\cdot}^2}}
$$
and using this fact into~\eqref{eq:z-error-expected} yields~\eqref{eq:MainInequality}.

%-------------------------------------------------------------------
\subsection{Proof of Proposition~\ref{pr:randomized-linear-convergence}}\label{app:randomized-linear_convergence}
The proof consists of the following steps: \textit{(i)} derive the auxiliary update in Algorithm~\ref{alg:robust-distributed-admm} as a perturbed, randomized affine operator, and characterize its properties; \textit{(ii)} bound the mean primal error for the quadratic approximation with the auxiliary error, which converges linearly; \textit{(iii)} extend the result to the general case.

%---------------------------------------
\paragraph*{\textit{(i)} Randomized perturbed affine operator}
Observe that, from~\eqref{eq:random-z_ij}, \eqref{eq:perturbed-affine-update} and, recalling the definition of $\B(k)$, we can write
\begin{align}
	&\z(k+1) = (\I - \B(k)) \z(k) + \nonumber \\ &\hspace{1.3cm} + \B(k) \left[ \T \z(k) + \uv + \o'(\x(k+1) - \x^*)\right] \nonumber \\
	&= \hat{\T}(k) \z(k) + \B(k) \left[ \uv + \o'(\x(k+1) - \x^*) \right] \label{eq:random-perturbed-update}
\end{align}
where $\hat{\T}(k) := \I - \B(k) (\I - \T)$. Let $\bar{\z} \in \fix(\Tpr)$, then since $\bar{\z} = \T \bar{\z} + \uv$, we have $\bar{\z} = \hat{\T}(k) \bar{\z} + \B(k) \uv$ for any $k \in \N$. Thus iterating~\eqref{eq:random-perturbed-update} and subtracting $\bar{\z}$ yields
\begin{equation}\label{eq:iterated-randomized-update}
\begin{split}
	&\z(k+1) - \bar{\z} = \prod_{\ell = 0}^{k} \hat{\T}(\ell) (\z(0) - \bar{\z}) + \\ &\qquad + \sum_{h = 0}^k \left( \prod_{\ell = h+1}^{k} \hat{\T}(\ell) \right) \B(h) \o'(\x(h+1) - \x^*)
\end{split}
\end{equation}
where by convention $\prod_{\ell = k+1}^{k} \hat{\T}(\ell) = \I$. Let us consider now the quadratic case that we have assuming \eqref{eq:taylor-expansion} holds true with the residual equal
to $0$. In this case~\eqref{eq:iterated-randomized-update} becomes 
\begin{equation}\label{eq:iterated-randomized-update-quad}
	\z(k+1) - \bar{\z} = \prod_{\ell = 0}^k \hat{\T}(\ell) (\z(0) - \bar{\z}).
\end{equation}
By Proposition~\ref{cr:convergence_lossy} we know that $\z(k+1)$ converges with probability one to a fixed point $\bar{\z}' \in \fix(\Tpr)$, in general different from $\bar{\z}$, which implies from~\eqref{eq:iterated-randomized-update-quad} that
\begin{equation}\label{eq:as-iterated-convergence}
	\bar{\z}' - \bar{\z} \stackrel{\mathrm{a.s.}}{=} \lim_{k \to \infty} \left( \z(k+1) - \bar{\z} \right) = \lim_{k \to \infty} \prod_{\ell = 0}^k \hat{\T}(\ell) (\z(0) - \bar{\z}).
\end{equation}
Notice that, given any two fixed points $\bar{\z}, \bar{\z}' \in \fix(\Tpr)$, it holds $\bar{\z}' - \bar{\z} \in \ker(\I - \T)$, thus for~\eqref{eq:as-iterated-convergence} to be true there must exist $c_1, \ldots, c_H$ random variables such that
$$
	\lim_{k \to \infty} \prod_{\ell = 0}^k \hat{\T}(\ell) (\z(0) - \bar{\z}) = \sum_{h = 1}^H c_h \v_h,
$$
where recall that $\ker(\I - \T) = \operatorname{span} \{ \v_1, \ldots, \v_H \}$.
The realizations of $c_h$, $h=1,\ldots,H$ depend on the realizations of $\hat{\T}(\ell)$, $\ell \in \N$ and on the initial condition $\z(0)$.
In general $\z(0) - \bar{\z} \not\in \ker(\I - \T)$, which implies that we are able to find the random vectors $\eps_1, \ldots \eps_H$ such that
\begin{equation}\label{eq:evolution-T}
	\lim_{k \to \infty} \prod_{\ell = 0}^k \hat{\T}(\ell) = \sum_{h = 1}^H \v_h \eps_h^\top \quad \text{and} \quad c_h = \eps_h^\top (\z(0) - \bar{\z}).
\end{equation}

%---------------------------------------
\paragraph*{\textit{(ii)} Mean error bound}
By iterating~\eqref{eq:iterated-randomized-update-quad}, taking the expectation, exploiting~\eqref{eq:primal-error-bound} and Jensen's inequality, we can write
\begin{align*}
	&\expval{\norm{\x(k+1) - \x^*}} \leq \zeta \sqrt{\expval{\norm{\A^\top (\z(k) - \bar{\z})}^2}} \\
	%&\qquad = \zeta \sqrt{(\z(0) - \bar{\z})^\top \expval{\norm{\A^\top \prod_{\ell = 0}^{k-1} \hat{\T}(\ell)}^2} (\z(0) - \bar{\z})} \\
	&\qquad = \zeta \sqrt{(\z(0) - \bar{\z})^\top \Dm(k) (\z(0) - \bar{\z})}
\end{align*}
where
$$
	\Dm(k) := \expval{ \bigg( \prod_{\ell = 0}^{k-1} \hat{\T}(\ell) \bigg)^\top \A \A^\top \bigg( \prod_{\ell = 0}^{k-1} \hat{\T}(\ell) \bigg) },
$$
if $k\geq 1$, otherwise $\Dm(0) =  \A \A^\top $. 
Therefore, with a simple recursive argument, we can characterize the bound for the primal error in terms of the evolution of the linear system
\begin{align}
	\Dm(k+1) &= \expval{\hat{\T}(k)^\top \Dm(k) \hat{\T}(k)} = \expval{\hat{\T}(0)^\top \Dm(k) \hat{\T}(0)} \nonumber \\
	&=: \mathcal{L}(\Dm(k)) \label{eq:linear-system}
\end{align}
with initial condition $\Dm(0) = \A \A^\top$, and where the second equality holds since, by Assumption~\ref{as:asynchronous-lossy-scenario}, the $\hat{\T}(k)$ are independent and identically distributed.

\begin{lemma}\label{lem:T-randomized-properties}
The eigenvalues of the operator $\mathcal{L}$ are all either strictly inside the unitary circle, or in $1$. In particular, let $\eps_1, \ldots, \eps_H$ be the random vectors such that~\eqref{eq:evolution-T} holds. Then the eigenspace of $\mathcal{L}$ relative to $1$ is $H^2$-dimensional and is generated by $\mathbb{E}[\eps_i \eps_j^\top]$, $i,j = 1, \ldots, H$.
\end{lemma}
\begin{proof}
Taking the limit for $k \to \infty$ of~\eqref{eq:linear-system} and exploiting the property~\eqref{eq:evolution-T} proved above yields, for any initial condition $\Dm(0)$:
\begin{align}
	\lim_{k \to \infty} \Dm(k) &= \lim_{k \to \infty} \mathcal{L}^k(\Dm(0)) \nonumber \\
	&= \lim_{k \to \infty} \expval{ \bigg( \prod_{\ell = 0}^{k-1} \hat{\T}(\ell) \bigg)^\top \Dm(0) \bigg( \prod_{\ell = 0}^{k-1} \hat{\T}(\ell) \bigg) } \nonumber \\
	&= \expval{ \bigg( \sum_{h = 1}^H \v_h \eps_h^\top \bigg)^\top \Dm(0) \bigg( \sum_{h = 1}^H \v_h \eps_h^\top \bigg) } \nonumber \\
	&= \sum_{i = 1}^H \sum_{j = 1}^H (\v_i^\top \Dm(0) \v_j) \expval{\eps_i \eps_j^\top}. \label{eq:lim-Delta}
\end{align}
This proves that the eigenspace of $\mathcal{L}$ relative to the eigenvalue $1$ is $H^2$ dimensional and it is generated by $\mathbb{E}[\eps_i \eps_j^\top]$, $i, j = 1, \ldots, H$.
\end{proof}

By Lemma~\ref{lem:T_properties} we have $\ker(\I - \T) \subset \ker(\A^\top)$, and so~\eqref{eq:lim-Delta} with initial condition $\Dm(0) = \A \A^\top$ implies that $\A \A^\top$ is orthogonal to the eigenspace generated by the eigenvectors relative to $1$. Thus we have $\lim_{k \to \infty} \mathcal{L}^k(\Dm(0)) = 0$, which proves that the primal error converges linearly to zero. The rate of convergence is characterized by the largest eigenvalue of the linear system $\mathcal{L}$ strictly inside the unitary circle, that is by $\bar{\gamma}_\mathrm{M}$ as defined in~\eqref{eq:bar-gamma}.

%---------------------------------------
\paragraph*{\textit{(iii)} General case}
By the results of point 2) we can write
\begin{equation}\label{eq:bound-with-ess}
	\expval{\norm{\A^\top \prod_{\ell = 0}^{k-1} \hat{\T}(\ell)}^2} \leq C' \bar{\gamma}_{\mathrm{M}}^k
\end{equation}
for some $C' > 0$. Moreover, in the general case, iterating~\eqref{eq:iterated-randomized-update} and exploiting the primal error bound
$$
	\expval{\norm{\x(k+1) - \x^*}} \leq \zeta \sqrt{\expval{\norm{\A^\top (\z(k) - \bar{\z})}^2}},
$$
we obtain \eqref{eq:MainInequality}. Now, using \eqref{eq:bound-with-ess} and the fact that $\B(k)$ are independent and identically distributed, we have 
\begin{align*}
	&\expval{\norm{\x(k+1) - \x^*}} \leq \zeta \norm{\z(0) - \bar{\z}} \sqrt{C' \bar{\gamma}_{\mathrm{M}}^k} + \\ 
	& \zeta \sum_{h = 0}^{k-1} \sqrt{C' \bar{\gamma}_{\mathrm{M}}^{k - h-1}} \, \expval{\norm{\B(0)}} \, \expval{\norm{\o'(\x(h+1) - \x^*)}}
\end{align*}
Similarly to the proof of Proposition~\ref{pr:local-linear-convergence}, we can argue that there exists a sequence of positive numbers $\{ \delta_k \}_{k \geq 1}$ such that $\delta_{k+1} \leq \delta_k$, $\lim_{k \to \infty} \delta_k = 0$ and
$$
	\norm{\o'(\x(k+1) - \x^*)} \leq \delta_{k+1} \norm{\x(k+1) - \x^*}.
$$
Hence we have the following inequality
\begin{align*}
	&\expval{\norm{\x(k+1) - \x^*}} \leq \zeta \norm{\z(0) - \bar{\z}} \sqrt{C' \bar{\gamma}_{\mathrm{M}}^k} + \\ &+ \zeta \sqrt{C'} \expval{\norm{\B(0)}} \sum_{h = 0}^{k-1} \sqrt{\bar{\gamma}_{\mathrm{M}}^{k - h -1}} \delta_{h+1} \expval{\norm{\x(h+1) - \x^*}}
\end{align*}
and with the same argument employed in Appendix~\ref{app:linear_convergence} the proof of Proposition~\ref{pr:randomized-linear-convergence} is complete. \cvd

\begin{remark}
Point \textit{(ii)} of this proof extends to the distributed optimization scenario the results reported in \cite{fagnani_randomized_2008} for the randomized consensus problem.
\end{remark}

%-------------------------------------------------------------------
\subsection{Proof of Proposition~\ref{pr:randomized-convergence-rate}}\label{app:bound-convergence-rate}
The idea is to introduce a matrix representation of $\mathcal{L}$ and then compute its largest eigenvalue inside the unitary circle.

Let $\vect(\cdot)$ be the vectorization operator that, given a matrix $\mathbold{M} \in \R^{K \times K}$, returns the vector $\vect(\mathbold{M}) \in \R^{K^2}$ having $[\mathbold{M}]_{i,j}$ in position $(i-1) K + j$. A useful property of $\vect$ is that for a triplet of matrices of suitable dimensions we can write $\vect(\mathbold{A B C}) = (\mathbold{C}^\top \otimes \mathbold{A}) \vect(\mathbold{B})$.

\noindent Vectorizing the linear system $\mathcal{L}$ we obtain
\begin{align*}
	\vect(\Dm(k+1)) &= \expval{\hat{\T}(0)^\top \otimes \hat{\T}(0)^\top} \vect(\Dm(k)) \\
	&= \L \vect(\Dm(k))
\end{align*}
where $\bar{\gamma}_{\mathrm{M}}$ of $\mathcal{L}$ coincides with the largest eigenvalue of $\L$ strictly inside the unitary circle.

Using Assumption~\ref{as:asynchronous-lossy-scenario} we now give an explicit formula for $\L$ in terms of $\T$ and the expectation of $\B(0)$. The symmetry of $\T$, see Remark~\ref{rem:symmetry-T}, implies that of $\hat{\T}(0)$ and thus of $\L$. Therefore, omitting the dependence on time in $\B(0)$, we have:
\begin{align*}
	\L &= \mathbb{E}[ (\I- \B+\B\T) \otimes (\I-\B+\B\T)] = \\
	&= \mathbb{E}\Big[\I\otimes \I-\I\otimes \B+\I\otimes \B\T-\B\otimes \I+\B\otimes \B+\\
	&\ \ -\B\otimes \B\T+\B\T\otimes \I-\B\T\otimes \B+\B\T\otimes \B\T\Big],
\end{align*}
and by linearity of the expectation we can focus on each term separately. The first term is clearly equal to itself, while we have $\mathbb{E}[\I\otimes \B] = \I\otimes \mathbb{E}[\B]$, and, similarly, $\mathbb{E}[\B\otimes \I] = \mathbb{E}[\B] \otimes \I$.

\noindent The remaining terms can be computed using the following property of the Kronecker product $(\mathbold{AC} \otimes \mathbold{BD}) = (\mathbold{A} \otimes \mathbold{B}) (\mathbold{C} \otimes \mathbold{D})$ for matrices $\mathbold{A,B,C,D}$ of suitable dimensions:
\begin{itemize}
	\item $\mathbb{E}[\I\otimes \B\T] = \mathbb{E}[(\I \otimes \B) (\I \otimes \T)] = (\I \otimes \mathbb{E}[\B]) (\I \otimes \T) = \I \otimes \mathbb{E}[\B]\T$,
	
	\item $\mathbb{E}[\B\T \otimes \I] = \mathbb{E}[(\B \otimes \I) (\T \otimes \I)] = (\mathbb{E}[\B] \otimes \I) (\T \otimes \I) = \mathbb{E}[\B]\T \otimes \I$,

	\item $\mathbb{E}[\B \otimes \B\T] = \mathbb{E}[(\B \otimes \B) (\I \otimes \T)] = \mathbb{E}[\B \otimes \B] (\I \otimes \T)$,
	
	\item $\mathbb{E}[\B\T \otimes \B] = \mathbb{E}[(\B \otimes \B) (\T \otimes \I)] = \mathbb{E}[\B \otimes \B] (\T \otimes \I)$,

	\item $\mathbb{E}[\B\T \otimes \B\T] = \mathbb{E}[(\B \otimes \B) (\T \otimes \T)] = \mathbb{E}[\B \otimes \B] (\T\otimes \T)$.
\end{itemize}
Summing and rearranging the terms (exploiting the properties of the Kronecker product) we thus prove Proposition~\ref{pr:randomized-convergence-rate}. \cvd

%--------------------------------------------------------------------------------------------------------------------------

\bibliographystyle{IEEEtran} % use IEEEtran.bst style
\bibliography{IEEEabrv,references}

%--------------------------------------------------------------------------------------------------------------------------

\vfill

\end{document}